\documentclass[11pt]{article}
\usepackage[margin=1in]{geometry}
\usepackage{amsmath,amssymb,amsthm,mathrsfs}
\usepackage{mathtools}
\usepackage{graphicx}
\usepackage[authoryear]{natbib}
\usepackage{enumitem}
\usepackage{setspace}
\usepackage{subfigure}
\usepackage{bbm}
\usepackage{bm}
\usepackage{booktabs}
\usepackage{caption}
\RequirePackage[colorlinks,citecolor=blue,urlcolor=blue]{hyperref}
\providecommand{\abs}[1]{\lvert#1\rvert}

\newtheorem{theorem}{Theorem}[section]
\newtheorem{lemma}{Lemma}[section]
\newtheorem{proposition}{Proposition}[section]
\newtheorem{corollary}{Corollary}

\newtheorem{example}{Example}
\newtheorem{remark}{Remark}[section]
\theoremstyle{definition}
\numberwithin{equation}{section}
\newtheorem{definition}{Definition}
\newtheorem{assumption}{Assumption}

\begin{document}
\noindent
\parskip=8pt
\baselineskip=16.8pt

\title{Impulse Control with Discontinuous Setup Costs: Discounted Cost Criterion\thanks{This work was supported in part by the National Natural Science Foundation of China under grant 11771432 and 11401566.}}

\author{Fen Xu\thanks{Institute of Applied Mathematics, Academy of Mathematics and Systems Science, Chinese
  Academy of Sciences, Beijing, 100190, China; School of Mathematical Sciences, University of Chinese Academy of Sciences, Beijing 100049, China  (xufen@amss.ac.cn).},\ \
 Dacheng Yao\thanks{Corresponding author. Institute of Applied Mathematics, Academy of Mathematics and Systems Science, Chinese
  Academy of Sciences, Beijing, 100190, China; Key Laboratory of Management, Decision and Information Systems, Chinese
  Academy of Sciences, Beijing, 100190, China
  (dachengyao@amss.ac.cn).},\ \  and
Hanqin Zhang\thanks{Department of Analytics $\&$ Operations, National University of Singapore, 119245, Singapore
(bizzhq@nus.edu.sg)}}

\maketitle
\begin{abstract}
This paper studies a continuous-review backlogged inventory model considered by \cite{HelmesStockbridgeZhu2015} but with discontinuous quantity-dependent setup cost for each order.
In particular, the setup cost is characterized by a two-step function and a higher cost would be charged once the order quantity exceeds a threshold $Q$. Unlike the optimality of $(s,S)$-type policy obtained by \cite{HelmesStockbridgeZhu2015} for continuous setup cost with the discounted cost criterion,
we find that, in our model, although some $(s,S)$-type policy is indeed optimal in some cases, the $(s,S)$-type policy can not always be optimal.
In particular, we show that there exist cases in which an $(s,S)$ policy is optimal for some initial levels but it is strictly worse than a generalized $(s,\{S(x):x\leqslant s\})$ policy for the other initial levels. Under $(s,\{S(x):x\leqslant s\})$ policy, it orders nothing for $x>s$ and orders up to level $S(x)$ for $x\leqslant s$, where $S(x)$ is a non-constant function of $x$. We further prove the optimality of such $(s,\{S(x):x\leqslant s\})$ policy in a large subset of admissible policies for those initial levels. Moreover, the optimality is obtained through establishing a more general lower bound theorem
which will also be applicable in solving some other optimization problems by the common lower bound approach.

\vspace{.25in}

\end{abstract}

\section{Introduction}
\label{sec:intro}
In this paper, we study the optimal ordering policy with the discounted cost criterion for a single-item inventory system in the presence of quantity-dependent setup cost and proportional cost as well as holding/backorder cost.
The inventory processes in the absence of control are modeled as solutions to a stochastic differential equation
\begin{equation}
\label{eq:X}
\mathrm{d} X(t)=-\mu\, \mathrm{d}t+\sigma\, \mathrm{d} B(t),\quad X(0)=x,
\end{equation}
where both $\mu$ and $\sigma$ are strictly positive constants.

The system manager replenishes the inventory from an outside supplier with unlimited items.
The presence of positive setup cost incurred by each order yields that the problem is an impulse control one and an ordering policy $\phi$ can be represented as a sequence of pairs $\{(\tau_n,\xi_n):n=0,1,2,\cdots\}$, in which $\tau_n$ denotes the $n$th ordering time and $\xi_n$ denotes the corresponding order quantity.
By convention, we set $\tau_0=0$ and let $\xi_0\geqslant0$ be the ordering amount at time zero ($\xi_0=0$ if no order is placed).
The inventory level process under any given impulse control policy therefore satisfies
\begin{equation}
\label{eq:Z}
Z(t)=x-\mu t +\sigma B(t)+\sum_{n=0}^{N(t)}\xi_n,
\end{equation}
where $N(t)=\max\{n\geqslant0:\tau_n\leqslant t\}$ denotes the cumulative ordering times in $(0,t]$.

Each order with quantity $\xi$ incurs a setup cost and a proportional cost $k\xi$. In particular, the setup cost is given by a two-step function
\begin{equation}
\label{eq:setup}
K(\xi)=K_1\mathbbm{1}_{\{0<\xi\leqslant Q\}}+K_2 \mathbbm{1}_{\{\xi>Q\}},
\end{equation}
where $K_2>K_1>0$ and $\mathbbm{1}_A$ is an indicator function with $\mathbbm{1}_A=1$ if $A$ is true and $\mathbbm{1}_A=0$ otherwise.
That means a high setup cost $K_2$ would be charged whenever the order quantity exceeds $Q$ and a low setup cost $K_1$ is incurred otherwise.
At the same time, the system continuously incurs a holding/backorder cost $g(z)$ when the inventory level is $z$, where $g(z)$ is a convex function.
Let $\beta>0$ be a discount factor.
Under an impulse control policy $\phi=\{(\tau_n,\xi_n):n=0,1,2,\cdots\}$ and an initial inventory level $x$, the expected present cost of holding/backorder and ordering is
\begin{align}
\label{eq:DC-cost}
\mathsf{DC}(x,\phi)=\mathbb{E}_x\Big[\int_0^{\infty} e^{-\beta t} g(Z(t))\,\mathrm{d}t+\sum_{n=0}^{\infty}e^{-\beta \tau_n}\big(K(\xi_n)+k\xi_n\big)\Big],
\end{align}
where $\mathbb{E}_x[\cdot]$ is the expectation conditioning on the initial inventory level $Z(0-)=X(0)=x$.
The objective is to find a policy $\phi^*$ to minimize \eqref{eq:DC-cost} over all ordering policies.

In the classical inventory models, a constant setup cost is usually assumed when an order is placed, and then an $(s, S)$-type policy turns out to be optimal; see e.g., \cite{Scarf1960} and \cite{Veinott1966} for periodic-review models, and \cite{Bather1966} and \cite{BensoussanLiuSethi2005} for continuous-review models. In many practical application, however, the setup costs are more complex and often depend on the order quantity. For example, when the setup costs arise from the labor costs for offloading task, depending on the size of order quantity, different amount of crews are assigned to the offloading and incur different setup costs (cf. \cite{Bigham1986} and  \cite{Caliskan-DemiragChenYang2012}).

The setup cost \eqref{eq:setup} was first studied in a periodic-review inventory model by \cite{Lippman1969}, which considered a subadditive ordering cost function including \eqref{eq:setup} as a special case and partially characterized the optimal ordering policy. Later, \cite{ChaoZipkin2008} considered setup cost \eqref{eq:setup} with $K_1=0$ and \cite{Caliskan-DemiragChenYang2012} considered \eqref{eq:setup} in a general case. Both of them studied finite-horizon periodic-review inventory models and
provided partial characterization of optimal ordering policies. The partial characterization indicates that the optimal ordering policies for periodic-review models with setup cost \eqref{eq:setup} would be very complicated.

Parallel to periodic-review inventory models, general ordering cost, including quantity-dependent setup cost, was recently studied in continuous-review inventory models.
In particular, \cite{HelmesStockbridgeZhu2015,HelmesStockbridgeZhu2017,HelmesStockbridgeZhu2018} studied continuous ordering cost functions
and proved the optimality of $(s,S)$-type policies for the inventory models with general diffusion process, under long-run average and discounted cost criterions respectively. Motivated by some real inventory problems, at the same time, discontinuous ordering cost began to attract more attention from both the optimal control community and the operations management community.
 \cite{HeYaoZhang2017} studied a general quantity-dependent setup cost and \cite{YaoChaoWu2017} studied a general piece-wise ordering cost, both in Brownian motion models. Furthermore, a very general ordering cost function was considered by  \cite{PereraJanakiramanNiu2017} in a deterministic EOQ model, and by \cite{PereraJanakiramanNiu2018} in an inventory model with renewal process demand.
Under the long-run average cost criterion, they all proved the optimality of $(s,S)$-type policies.
Thus, so far the optimality of $(s,S)$-type policy has been obtained under the long-run average cost criterion with continuous/discontinuous ordering cost function and under the discounted cost criterion but only with the continuous ordering cost function. {\it So a natural question is  whether the discontinuous ordering cost still admits an $(s,S)$-type policy to be optimal for the discounted cost criterion.}

When the initial inventory level is strictly limited to be nonnegative,  \cite{Jia2016} proved that an $(s,S)$-type policy is optimal.
Generally speaking, the total cost depends on the initial level and the initial level plays a crucial role in determining the optimal policies for the discounted cost models (compared with the average cost models); see e.g., \cite{HelmesStockbridgeZhu2015}. Unlike the nonnegative initial level case considered by \cite{Jia2016}, and the continuous ordering cost case investigated by  \cite{HelmesStockbridgeZhu2015}, this paper, however, shows that the optimal policies with arbitrary initial level and discontinuous ordering cost are very complicated and in some cases, any $(s,S)$-type policy is not always optimal.

To identify the optimal policy, a new approach is developed as follows:  First we characterize two ``good"  $(s,S)$-type policies, embodied by $(s_1,S_1)$ policy and $(s_2,S_2)$ policy, respectively. The $(s_1,S_1)$ ($(s_2,S_2)$) policy is obtained by minimizing the inventory problems \eqref{eq:DC-cost} limited in $(s,S)$-type policies but with order quantity constraint $\xi\leqslant Q$ ($\xi\geqslant Q$) and constant setup cost $K_1$ ($K_2$).
We then show that, in our original problem, if $(s_2,S_2)$ policy is better than $(s_1,S_1)$ policy for $x\geqslant \max\{s_1,s_2\}$, $(s_2,S_2)$ policy would be optimal for all initial inventory level $x\in \mathbb{R}$. Otherwise, $(s_1,S_1)$ policy is optimal only for initial inventory levels $x\in[S_1-Q,\infty)$ but is strictly worse than a generalized $(s_1,\{S^*(x):x\leqslant s_1\})$ policy for some initial levels in $(-\infty,S_1-Q)$. Under $(s_1,\{S^*(x):x\leqslant s_1\})$ policy, it orders nothing for initial level $x>s_1$ and orders up to level $S^*(x)$ for initial level $x\leqslant s_1$, where $S^*(x)$ equals $S_1$ for $S_1-Q<x\leqslant s_1$,  equals $x+Q$ for $\underline{s}\leqslant x\leqslant S_1-Q$, and equals $\bar{S}$ for $x<\underline{s}$.  The parameters $\underline{s}$ and $\bar{S}$ are determined by certain equations and satisfy $s_1-Q\leqslant \underline{s}\leqslant S_1-Q$ and $\bar{S}\geqslant S_1$.
We further show that this generalized policy is optimal in a large subset of policies when certain conditions are satisfied.

More specifically, in the first step of our approach, to characterize the $(s_i,S_i)$ policies in the constrained Brownian control problems, we solve the KKT conditions of two independent constrained nonlinear optimization problems, rather than solving the Brownian control problems with order quantity constraints directly such as in literature; see e.g., \cite{OrmeciDaiVandeVate2008} and  \cite{Jia2016}.
Using this deterministic optimization method, the connection between order quantity constraints and smooth-pasting conditions can be clearly established.

To prove the optimality of the selected policy, we generalize the lower bound theorem (also called verification theorem) in literature (cf.,  \cite{DaiYao2013b},  \cite{HarrisonSellkeTaylor1983}, and \cite{OrmeciDaiVandeVate2008}), which requires the value/cost function should be of $\mathcal{C}^1$ at all points, such that the results still hold even when the value/cost function is not of $\mathcal{C}^1$  at finite points.
The $\mathcal{C}^1$ condition is also required in the measure approach (cf.  \cite{HelmesStockbridgeZhu2017}) and the quasi-variational inequalities (QVI) approach (cf.  \cite{BensoussanLiuSethi2005} and  \cite{HelmesStockbridgeZhu2018}).
Thus, the generalized lower bound theorem established in this paper will also be applicable to other optimal control problems whose value/cost functions are not of $\mathcal{C}^1$.

The rest of this paper is organized as follows. In \S \ref{sec:model-preliminary}, we formulate our model
and give some preliminary results about two candidate $(s,S)$-type policies. In  \S \ref{sec:optimal}, we show our main results, i.e., Theorems \ref{thm1}, \ref{thm2}, \ref{thm3}, and \ref{thm4}. In \S \ref{sec:LBT}, we provide the generalized lower bound theorem, and using it,  we prove our main results in \S \ref{sec:proof}. In \S \ref{sec:condition-Q}, we
discuss a condition in Theorem \ref{thm4} by numerical studies and further show that this condition always holds when $Q$ is large.
Finally, we conclude this paper in \S \ref{sec:conclusion}, and provide some supplementary proofs in Appendix \ref{sec:appendix}.

\section{Model formulation and preliminaries}
\label{sec:model-preliminary}

\subsection{Model formulation}
We assume that $B=\{B(t):t\geqslant0\}$ in \eqref{eq:X} is a standard Brownian motion defined on a filtered probability space $(\Omega,\mathcal{F},\mathbb{F},\mathbb{P})$ with filtration $\mathbb{F}=\{\mathcal{F}(t):t\geqslant0\}$.
The generator of process $X$ in \eqref{eq:X} is
\begin{equation}
\label{eq:Gamma}
\Gamma f (x) = \frac{1}{2}\sigma^2f''(x)-\mu f'(x).
\end{equation}

An impulse control policy $\phi=\{(\tau_n,\xi_n):n=0,1,2,\cdots\}$ is said to be \emph{admissible} if $\tau_n$ is an $\mathbb{F}$-stopping time, $\xi_n$ is $\mathcal{F}_{\tau_n}$-measurable, and
\[
\mathbb{E}_x\Big[\sum_{n=0}^{N(t)}\xi_n\Big]<\infty\quad \text{for any $t\geqslant0$},
\]
which ensure the controlled process $Z$ in \eqref{eq:Z} is a semimartingale.
Let $\mathcal{P}$ be the set of all admissible policies.  The objective is to find an admissible policy $\phi^*$ such that
\begin{equation}
\label{eq:phi-star}
\mathsf{DC}(x,\phi^*)=\inf_{\phi\in\mathcal{P}} \mathsf{DC}(x,\phi).
\end{equation}

In this paper, the holding/shortage cost function $g(\cdot)$ satisfies the following assumption.
\begin{assumption}\label{ass:h}
\begin{itemize}
\item[(A1)]$g(\cdot)$ is convex with $g(0) = 0$;
\item[(A2)]$g(\cdot)$ is twice continuously differentiable except at 0;
\item[(A3)]$g'(x)<0$ for $x < 0$ and  $g'(x)>0$ for $x> 0$;
\item[(A4)] $\lim_{x\rightarrow -\infty}[g'(x)+\beta k] < -\beta K_1/Q$; 
\item[(A5)]$g(x)$ is polynomially bounded, i.e., there exist a positive integer $n$ and two positive numbers $a,b$ such that $g(x)\leqslant a+b\abs{x}^n$ for all $x\in \mathbb{R}$.
\end{itemize}
\end{assumption}

\begin{remark}
\label{rem:h}
($a$) We explain Assumption (A4) as follows.
Note that the convexity of $g$ implies that $\lim_{x\rightarrow -\infty}(-g'(x))/\beta$ is an upper bound of the present value of the backorder cost per unit from now to infinity, and $K_1/Q+k$ is  a lower bound of the ordering cost per unit with order quantity $\xi\in(0,Q]$. If $\lim_{x\rightarrow -\infty}(-g'(x))/\beta \leqslant K_1/Q+k$, it will never be optimal to place any order with $\xi\in(0,Q]$ and then
we only need consider the policy with $\xi\in(Q,\infty)$. In order to facilitate the analysis and avoid this trivial case, we assume $\lim_{x\rightarrow -\infty}(-g'(x))/\beta > K_1/Q+k$, i.e., $\lim_{x\rightarrow -\infty}[g'(x)+\beta k] < -\beta K_1/Q$.
\\
($b$) We can get that $g'(x)$ is also polynomially bounded. To see this, noting that $g(x)$ is convex, we have that for any $x\neq 0$,
\[
g(x+1)\geqslant g(x)+g'(x)\quad\text{and}\quad g(x-1)\geqslant g(x)-g'(x),
\]
which, together with the nonnegativity of $g$ (which can be obtained by (A1) and (A3)), imply
\[
-g(x-1)\leqslant g'(x)\leqslant g(x+1),
\]
Therefore, $g'(x)$ is also polynomially bounded.
\\
($c$) Both piecewise linear cost $g(x)=hx^+-px^-$ with  $-p+\beta k<-\beta K_1/Q$ and quadratic cost $g(x)=\alpha x^2$ with $\alpha>0$ satisfy (A1)-(A5).
\qed
\end{remark}

Further, the setup cost function $K(\cdot)$ in \eqref{eq:setup} satisfies the following assumption.
\begin{assumption}\label{ass:K}
$K_1< K_2\leqslant 2K_1$.
\end{assumption}

Note that $K_1<K_2$ means that a higher setup cost would be incurred if the order quantity exceeds $Q$ (cf.  \cite{Caliskan-DemiragChenYang2012} and \cite{ChaoZipkin2008}). We explain the assumption $K_2\leqslant 2K_1$ as follows. If $2K_1<K_2$, the cost of placing two orders with the same amount $Q'$, satisfying $0<Q'\leqslant Q$ and $2Q'>Q$, is $2K_1+2kQ'$, which is strictly lower than $K_2+2kQ'$, the cost of placing one order with amount $2Q'$. Thus, a policy that places multiple orders at the same time may be optimal. In this paper, we do not consider the case placing multiple orders at a single time point, thus we assume $2K_1\geqslant K_2$. In fact, the setup cost function $K(\cdot)$ satisfying Assumption \ref{ass:K} is a subadditive function\footnote{A function $f$ is subadditive in domain $D$ if $f(x+y)\leqslant f(x)+f(y)$ for any $x,y\in D$.}. This assumption has been given in  \cite{Lippman1969}.

\subsection{Preliminaries}
\label{sec:sS}

In this subsection, we study two independent models,  embodied by $\mathcal{M}_i$, $i=1,2$, which
have the same setting as our original model except the following:
\begin{itemize}
\item[]\noindent Model $\mathcal{M}_1$: order quantity is constrained in $[0,Q]$ with setup cost $K_1$;
\item[]\noindent Model $\mathcal{M}_2$: order quantity is constrained in $[Q,\infty)$ with setup cost $K_2$.
\end{itemize}
Then we obtain a ``good" $(s,S)$ policy for each model, and these two $(s,S)$ policies would be used to characterize our optimal policy.

First, if the setup cost in our original model is a constant $K_1$ or $K_2$  for any order quantity, our problem becomes a classical inventory model, that admits a $(s,S)$ policy to be optimal. This has been studied in \cite{benkherouf2007} and \cite{Sulem1986} (and \cite{HelmesStockbridgeZhu2015} for a more general demand model).
For Model $\mathcal{M}_i$, let $v_i^{(s,S)}(x)$ denote the expected discounted cost for initial inventory level $x\in\mathbb{R}$ under a given $(s,S)$ policy.
Under given $(s,S)$ policy, if the current inventory level is larger than $s$, i.e., $x>s$, no order is placed before the inventory level process hits $s$, thus
\begin{align*}
v_i^{(s,S)}(x)&=\mathbb{E}_x\Big[\int_0^{\tau(s)} e^{-\beta t} g(Z(t))\,\mathrm{d}t +e^{-\beta \tau(s)} v_i^{(s,S)}(Z(\tau(s)))\Big]\quad \text{for $x>s$},
\end{align*}
where $\tau(s)$ is the first hitting time of process $Z$ at level $s$. Thus,  $v_i^{(s,S)}(x)$ satisfies the following ordinary differential equation and boundary condition
(cf. Page 127 of \cite{Sulem1986} and Page 170 of  \cite{WuChao2014})
\begin{align}
&\Gamma v_i^{(s,S)}(x)-\beta v_i^{(s,S)}(x)+g(x)=0,\quad \text{for $x>s$,}\quad \text{and}\label{eq:ode-1}\\
& \lim_{x\to\infty}e^{-\alpha x}v_i^{(s,S)}(x)=0,\quad \text{$\forall \alpha>0$},\label{eq:limit-condition}
\end{align}
where $\Gamma$ is defined in \eqref{eq:Gamma}.
Further, once the inventory level hits $s$, it will jump to $S$ by placing an order with quantity $S-s$, incurring an ordering cost $K_i+k\cdot(S-s)$, that means
\begin{align}
\label{eq:boundary-condition}
v_i^{(s,S)}(s)&=v_i^{(s,S)}(S)+K_i+k\cdot(S-s).
\end{align}
Then the expected discounted cost $v_i^{(s,S)}(x)$  can be solved by \eqref{eq:ode-1}-\eqref{eq:boundary-condition}, that is,
\begin{align}
v_i^{(s,S)}(x)&=
\frac{2}{\sigma^2}\frac{1}{\lambda_1+\lambda_2}\Big[\Lambda_1(x) +\Lambda_2(x)-\frac{1}{\lambda_2^2}A_i(s,S) e^{-\lambda_2x}\Big]\quad\text{for $x>s$, and }\label{eq:solution-vi}\\
v_i^{(s,S)}(x)&=v_i^{(s,S)}(S)+K_i+k\cdot(S-x) \quad \text{for $x\leqslant s$},\label{eq:v=v+K}
\end{align}
where
\[
\Lambda_1(x)=\int_x^{\infty} e^{\lambda_1(x-y)}g(y)\,\mathrm{d}y,\quad
\Lambda_2(x)=\int_0^x e^{-\lambda_2(x-y)}g(y)\,\mathrm{d}y,
\]
and
\begin{align}
\label{eq:A(s,S)}
A_i(s,S)=\frac{\lambda_2^2}{e^{-\lambda_2 S}-e^{-\lambda_2 s}}
\Big[\sum_{j=1}^2\big(\Lambda_j(S)-\Lambda_j(s)\big)
+\frac{\sigma^2(\lambda_1+\lambda_2)}{2}\big(K_i+k\cdot(S-s)\big)\Big],
\end{align}
with $\lambda_1=\big(\mu+\sqrt{\mu^2+2\beta \sigma^2}\big)/\sigma^2>0$ and $\lambda_2=\big(-\mu+\sqrt{\mu^2+2\beta \sigma^2}\big)/\sigma^2>0$.

The cost function $v_i^{(s,S)}(x)$ in \eqref{eq:solution-vi} depends on policy parameters $s$ and $S$ only through $A_i(s,S)$
and is strictly decreasing in $A_i(s,S)$. Thus, to minimize the cost of Model $\mathcal{M}_i$ within $(s,S)$-type policies, we consider the following
two constrained optimization problems:
\begin{equation*}
\mathcal{OP}_1:\ A_1^*:= \max_{0<S-s\leqslant Q} A_1(s,S)\quad \text{and}\quad \mathcal{OP}_2:\ A_2^*:= \max_{S-s\geqslant Q} A_2(s,S).
\end{equation*}
Consider $\mathcal{OP}_1$, define the Lagrange function $\mathcal{L}_1(s,S,\eta_1)=-A_1(s,S)+\eta_1 (S-s-Q)$ with the Lagrangian multiplier $\eta_1$\footnote{Since $A_1(s,S)\to-\infty$ as $S-s\to 0$, the local maximizer can not be $s=S$, then we did not put the constraint $S-s>0$ into the Lagrange function and just put it into the KKT conditions}. Then the KKT conditions are
\begin{align}
\frac{\partial A_1(s,S)}{\partial s}+\eta_1&=0,\label{eq:KKTcondition-1}\\
\frac{\partial A_1(s,S)}{\partial S}-\eta_1&=0,\label{eq:KKTcondition-2}\\
\eta_1((S-s)-Q)&=0,\label{eq:KKTcondition-3}\\
(S-s)-Q&\leqslant 0,\label{eq:KKTcondition-4}\\
s-S&<0,\label{eq:KKTcondition-5}\\
\eta_1&\geqslant0.\label{eq:KKTcondition-6}
\end{align}
Since the constraint functions are linear functions (i.e., affine functions), the regularity conditions are satisfied. Thus, from the classical nonlinear optimization theory (cf. \cite{Bertsekas2016}), any local maximizer must satisfy the above KKT conditions.
Similarly, the Lagrange function of  $\mathcal{OP}_2$ is $\mathcal{L}_2(s,S,\eta_2)=-A_2(s,S)+\eta_2 (Q-(S-s))$
and the KKT conditions are
\begin{align}
&\frac{\partial A_2(s,S)}{\partial s}-\eta_2=0,\quad \frac{\partial A_2(s,S)}{\partial S}+\eta_2=0,\label{eq:KKT2-1}\\
&\eta_2(Q-(S-s))=0,\quad Q-(S-s)\leqslant 0, \quad \text{and}\quad \eta_2\geqslant0.\label{eq:KKT2-2}
\end{align}
Define
\begin{align}
v_{A_i}(x)&=\frac{2}{\sigma^2}\frac{1}{\lambda_1+\lambda_2}\Big[\Lambda_1(x) +\Lambda_2(x)-\frac{1}{\lambda_2^2}A_i e^{-\lambda_2x}
\Big]\quad \text{ for $x\in\mathbb{R}$ and $A_i\in\mathbb{R}$,  and }\label{eq:v-Ai}\\
v_i(x)&=v_{A_i^*}(x).\nonumber
\end{align}
Let
\begin{align}
\underline{A}&\triangleq\lambda_2\int_{-\infty}^0e^{\lambda_2y}g'(y)\,\mathrm{d}y=g'(0-)-\int^{0}_{-\infty}e^{\lambda_2y}g''(y)\,\mathrm{d}y \quad\text{and}\label{eq:underline-A}\\
\bar{A}&\triangleq\lambda_1\int_{0}^{\infty}e^{-\lambda_1 y}g'(y)\,\mathrm{d}y=g'(0+)+\int_{0}^{\infty}e^{-\lambda_1 y}g''(y)\,\mathrm{d}y,\label{eq:bar-A}
\end{align}
where we used integration by parts. Note that $\underline{A}<0<\bar{A}$ from Assumption \ref{ass:h} (A1) and (A3).
First, we have the following lemma, which is crucial to prove our main results.
The proof is postponed to  Appendix \ref{app:proof-lem-sS-A}.
\begin{lemma}
\label{lem:sS-A}
$(a)$ $\underline{A}<A_i^*<\bar{A}$, $i=1,2$.\\
$(b)$ For any given $A_i$ with $A_i\in(\underline{A},\bar{A})$, function $v_{A_i}'(x)$ is strictly quasi-convex, having a unique minimizer $x_{A_i}^*<0$.\\
$(c)$ There exists a unique finite solution $(s_i,S_i)$ with $s_i<x_i^*<S_i$ $($where $x_i^*:=x_{A_i^*}^*$$)$ of the KKT conditions \eqref{eq:KKTcondition-1}-\eqref{eq:KKTcondition-6} or \eqref{eq:KKT2-1}-\eqref{eq:KKT2-2} such that $A_i(s_i,S_i)\in (\underline{A},\bar{A})$. Further,
\begin{equation}
\label{eq:A-star}
A_i^*=A_i(s_i,S_i)\quad\text{and}\quad v_i(x)=v_{A_i(s_i,S_i)}(x).
\end{equation}
$(d)$ $v_1'(s_1) =v_1'(S_1)=-k$ if $S_1-s_1<Q$ and $v_1'(s_1)=v_1'(S_1)\leqslant -k$ if $S_1-s_1=Q$.\\
$(e)$ $v_2'(s_2)=v_2'(S_2)=-k$ if $S_2-s_2>Q$ and $v_2'(s_2)=v_2'(S_2)\geqslant -k$ if $S_2-s_2=Q$.\\
$(f)$ $g'(x)+\beta k<0$ for $x\leqslant x_1^*$ and  when $S_2-s_2>Q$, $g'(x)+\beta k<0$ for $x\leqslant x_2^*$.
\end{lemma}

\begin{remark}
\label{rem:vi}
$(a)$ It follows from Lemma \ref{lem:sS-A} $(b)$-$(c)$ that $v_i'$ is strictly quasi-convex, i.e.,
$v_i''(x)<0$ for $x<x_i^*$ and $v_i''(x)>0$ for $x>x_i^*$.\\
$(b)$ When $A_1^*>A_2^*$, we can get $s_1>s_2$. This is because, when $A_1^*> A_2^*$,  it follows from
\begin{equation}
\label{eq:vi-derivate}
v_i'(x)=\frac{2}{\sigma^2}\frac{1}{\lambda_1+\lambda_2}\Big[\lambda_1\Lambda_1(x) -\lambda_2\Lambda_2(x)+\frac{1}{\lambda_2}A_i^* e^{-\lambda_2x}
\Big]
\end{equation}
that $v_1'(s_2)>v_2'(s_2)$, which together with Lemma \ref{lem:sS-A} $(d)$-$(e)$, yields that
$v_1'(s_2)>v_2'(s_2)\geqslant-k\geqslant v_1'(s_1)$. Then, the quasi-convexity of $v_i'$ $($Remark \ref{rem:vi} $(a)$$)$ immediately implies $s_1>s_2$.
\end{remark}
From \eqref{eq:solution-vi}-\eqref{eq:v=v+K}, \eqref{eq:v-Ai}, and \eqref{eq:A-star},
we know that under $(s_i,S_i)$ policy, the associated cost \eqref{eq:DC-cost} for our original problem is
\begin{equation}
\label{eq:DC}
\mathsf{DC}(x,(s_i,S_i))=
\begin{cases}
v_i(x) & \text{for $x>s_i$},\\
v_i(S_i)+K(S_i-x)+k\cdot(S_i-x) & \text{for $x\leqslant s_i$}.
\end{cases}
\end{equation}

In Lemma \ref{lem:sS-A} ($d$)-($e$), $v_i'(s_i) =v_i'(S_i)=-k$ holds when the constraints are not tight. This is consistent with the literature without any order constraints; see \cite{benkherouf2007}. In fact, $v_i'(s_i) =v_i'(S_i)=-k$, called smooth-pasting conditions, are used to characterize the optimal policy parameters $(s_i,S_i)$ for the model without order constraints. However, Lemma \ref{lem:sS-A} ($d$)-($e$) also show that the smooth-pasting conditions no longer hold when the constraints are tight.

In the following section, we will show the performance of these two policies, which are indeed optimal in some cases.

\section{Main results}
\label{sec:optimal}

In this section, with Lemma \ref{lem:sS-A} in hand, we present our main results. We show that if $A_1^*\leqslant A_2^*$, $(s_2,S_2)$ policy is optimal for our problem \eqref{eq:phi-star} (Theorem \ref{thm1}), and that if $A_1^*>A_2^*$, $(s_1,S_1)$ policy is optimal only when the initial inventory level $x$ is in $[S_1-Q,\infty)$ (Theorem \ref{thm2}), but it is strictly worse than a generalized $(s_1,\{S^*(x):x\leqslant s_1\})$ policy for some $x\in(-\infty,S_1-Q)$ (Theorem \ref{thm3}).
Finally, we, in Theorem \ref{thm4}, show that if $A_1^*>A_2^*$, such generalized policy is optimal in a subset of admissible policies for $x\in(-\infty,S_1-Q)$.

First, it follows from \eqref{eq:v-Ai} and \eqref{eq:A-star}, that if $A_1^*\leqslant A_2^*$,
\begin{equation}
\label{eq:v1>v2}
v_1(x)\geqslant v_2(x)\quad \text{for $x\in\mathbb{R}$}.
\end{equation}
Furthermore, \eqref{eq:DC} has shown that for the initial inventory level $x\in[s_i,\infty)$, $v_i(x)$ denotes the discounted cost  under $(s_i,S_i)$ policy. Thus, \eqref{eq:v1>v2} implies that for any initial level $x\in[s_1\vee s_2,\infty)$, the discounted cost under $(s_2,S_2)$ policy is lower than that under  $(s_1,S_1)$ policy. The following Theorem \ref{thm1} will show that when $A_1^*\leqslant A_2^*$,  $(s_2,S_2)$ policy indeed is optimal for any $x\in\mathbb{R}$. The proof of this theorem is provided in Section \ref{sec:proof-thm1}.

\begin{theorem}
\label{thm1}
If $A_1^*\leqslant A_2^*$, $(s_2,S_2)$ policy is optimal in  $\mathcal{P}$ for any initial level $x\in\mathbb{R}$, and the optimal cost is
\begin{equation}
\label{eq:V2}
\mathsf{DC}(x,(s_2,S_2))=V_2(x):=
\begin{cases}
v_2(x), & \text{for $x\geqslant s_2$},\\
v_2(S_2)+K_2+k\cdot(S_2-x), & \text{for $x<s_2$}.
\end{cases}
\end{equation}
\end{theorem}

By $K_1<K_2$, the optimal cost for system $\mathcal{M}_1$ will be less than the one for system $\mathcal{M}_2$ if the optimal ordering quantity $S_1-s_1$ is not  bounded by $Q$ in system $\mathcal{M}_1$. Hence we can find that $A_1^*\leqslant A_2^*$ only happens when $S_1-s_1=Q$.
Thus, Theorem \ref{thm1} tells us when it happens, $(s_2,S_2)$ policy is better than $(s_1,S_1)$ policy for any initial inventory level $x\in\mathbb{R}$. This ensures the optimality of $(s_2,S_2)$ policy.

Consider the other case when $A_1^*>A_2^*$, we have
\[
v_1(x)< v_2(x)\quad \text{for $x\in\mathbb{R}$.}
\]
Note that $s_1>s_2$ (see Remark \ref{rem:vi} ($b$)), thus for any initial level $x\in[s_1,\infty)$, the cost under $(s_1,S_1)$ policy is lower than that under $(s_2,S_2)$ policy.
It seems that $(s_1,S_1)$ policy is a ``good" policy, especially when the initial inventory level is in $[s_1,\infty)$.
In fact, the following theorem shows that $(s_1,S_1)$ policy is optimal for any $x\in[S_1-Q,\infty)$.
The proof is given in Section \ref{sec:proof-thm2}.

\begin{theorem}
\label{thm2}
If $A_1^*>A_2^*$, for any initial level $x\in[S_1-Q,\infty)$, $(s_1,S_1)$ policy is optimal in $\mathcal{P}$, and the optimal cost is
\begin{equation*}
\mathsf{DC}(x,(s_1,S_1))=
\begin{cases}
v_1(x), & \text{for $x\geqslant s_1$},\\
v_1(S_1)+K_1+k\cdot(S_1-x), & \text{for $S_1-Q\leqslant x<s_1$}.
\end{cases}
\end{equation*}
\end{theorem}

One would speculate $(s_1,S_1)$ policy to be the optimal policy when $x<S_1-Q$ naturally, which is, however, not the truth here.
For example, consider the initial inventory level $x_1=S_1-Q$ and $x_2=S_1-Q-\varepsilon$ for small $\varepsilon>0$.
Under $(s_1,S_1)$ policy, we have
\[
K(S_1-x_1)=K_1\quad\text{and}\quad K(S_1-x_2)=K_2,
\]
i.e., setup cost changes from $K_1$ to $K_2$ abruptly while initial inventory level $x$ varies slightly from $S_1-Q$ to $S_1-Q-\varepsilon$. Motivated by this, we may construct a policy which still places an order with quantity $Q$ even for initial level $x_2=S_1-Q-\varepsilon$, incurring setup cost $K_1$, and then the order-up-to inventory level becomes $S_1-\varepsilon$.
This action does not belong to $(s_1,S_1)$ policy but seems to perform better in some sense. To construct this kind of policy, we first introduce two parameters in the following lemma.

\begin{lemma}
\label{lem:bar-S}
When $A_1^*>A_2^*$, \\
$(a)$ there exists a unique solution, denoted by $\bar{S}$, in $[S_1,\infty)$ such that $v_1'(x)=-k$; and further\\
$(b)$ there exists a unique solution, denoted by $\underline{s}$, in $[s_1-Q, S_1-Q]$ such that
\begin{equation}
\label{eq:underline-s}
\mathcal{H}(x)=0,
\end{equation}
where
\begin{equation*}
\mathcal{H}(x)=\big(v_1(x+Q)+K_1+kQ\big)-\big(v_1(\bar{S})+K_2+k\cdot(\bar{S}-x)\big).
\end{equation*}
In particular, if $v_1'(s_1)=v_1'(S_1)=-k$, we have
\begin{equation}
\label{eq:underline-s<}
\underline{s}<S_1-Q.
\end{equation}
\end{lemma}

Using these two parameters, we define a generalized  policy, named $(s_1,\{S^*(x):x\leqslant s_1\})$  policy.

\begin{definition}
\label{def-policy}
A policy is called $(s_1,\{S^*(x):x\leqslant s_1\})$ policy if  it orders nothing for $x>s_1$, and orders up to level $S^*(x)$ for $x\leqslant s_1$, where
\begin{align*}
\label{eq:S-star}
S^*(x): =
\begin{cases}
S_1 & \text{for $x\in(S_1-Q,s_1]$},\\
x+Q &\text{for $x\in[\underline{s},S_1-Q]$},\\
\bar{S} &\text{for $x\in(-\infty,\underline{s})$},
\end{cases}
\end{align*}
with $\underline{s}$ and $\bar{S}$ being defined in Lemma {\rm\ref{lem:bar-S}}.
\end{definition}

Under $(s_1,\{S^*(x):x\leqslant s_1\})$ policy,
\begin{equation}
\label{eq:valfun}
\mathsf{DC}(x,(s_1,\{S^*(x):x\leqslant s_1\}))=
  \begin{cases}
  v_{1}(x), &\text{for $x \in (s_1,\infty)$},\\
  v_{1}(S_1)+K_{1}+k(S_1-x), &\text{for $x\in (S_1-Q,s_1]$},\\
  v_{1}(x+Q)+K_{1}+kQ, &\text{for $x\in [\underline{s},S_1-Q]$},\\
  v_{1}(\bar{S})+ K_{2}+k(\bar{S}-x), &\text{for $x\in(-\infty,\underline{s})$}.
\end{cases}
\end{equation}
We next compare the actions under both $(s_1,S_1)$ policy and $(s_1,\{S^*(x):x\leqslant s_1\})$ policy.
First, for $x\in(s_1,\infty)$, both policies order nothing and if $x\in[S_1-Q,s_1]$, both order items up to level $S_1$.
Furthermore, for $x\in[\underline{s},S_1-Q)$, $(s_1,S_1)$ policy orders items up to level $S_1$ while $(s_1,\{S^*(x):x\leqslant s_1\})$  policy orders items up to level $x+Q$, which is strictly smaller than $S_1$. Finally, we check the case when $x\in(-\infty,\underline{s})$ by two cases: $v_1'(s_1)=v_1'(S_1)=-k$ and $v_1'(s_1)=v_1'(S_1)<-k$ (from Lemma \ref{lem:sS-A} ($d$), $v_1'(s_1)=v_1'(S_1)\leqslant-k$ always holds).

If $v_1'(s_1)=v_1'(S_1)=-k$, it follows from the definition of $\bar{S}$ in Lemma \ref{lem:bar-S} (a) that
\begin{equation}
\label{eq:S=S}
\bar{S}=S_1,
\end{equation}
then both policies take the same action and order items up to level $S_1$.

If $v_1'(s_1)=v_1'(S_1)<-k$, for $x\in(-\infty,\underline{s})$, since $v_1'(S_1)<-k=v_1'(\bar{S})$, we have
\[
\bar{S}> S_1,
\]
and then $(s_1,S_1)$ policy orders items up to level $S_1$ while $(s_1,\{S^*(x):x\leqslant s_1\})$ policy orders items up to level $\bar{S}$.

Based on the above comparison between these two policies, we can see that $(s_1,\{S^*(x):x\leqslant s_1\})$ policy improves $(s_1,S_1)$ policy in two scenarios and thus is better than $(s_1,S_1)$ policy.
First, when $x$ is lower than $S_1-Q$ but not too low, i.e., $x\in[\underline{s},S_1-Q)$, unlike $(s_1,S_1)$ policy,
$(s_1,\{S^*(x):x\leqslant s_1\})$ policy still orders quantity $Q$ with setup cost $K_1$, avoiding abrupt jump from $K_1$ to $K_2$ when $x$ changes from $S_1-Q$ to $S_1-Q-\varepsilon$. Second, when $x$ is too low, i.e., $x\in(-\infty,\underline{s})$, $(s_1,S_1)$ policy orders up to $S_1$
while $(s_1,\{S^*(x):x\leqslant s_1\})$ policy orders up to $\bar{S}$, and both incur setup cost $K_2$.
However, $\bar{S}$ satisfies $v_1'(\bar{S})=-k$ and thus it minimizes $v_1(y)+K_2+k\cdot(y-x)$,
but it is not necessarily for $S_1$.
In summary, we have the following theorem, whose proof is in Section \ref{sec:proof-thm3}.
\begin{theorem}
\label{thm3}
If $A_1^*>A_2^*$, we have the following result.\\
$(a)$ If $v_1'(s_1)=v_1'(S_1)=-k$, for any initial level $x\in(\underline{s},S_1-Q)$, $(s_1,S_1)$ policy is strictly worse than $(s_1,\{S^*(x):x\leqslant s_1\})$ policy.\\
$(b)$ If $v_1'(s_1)=v_1'(S_1)<-k$, for any initial level $x\in(-\infty, S_1-Q)\setminus \{\underline{s}\}$, $(s_1,S_1)$ policy is strictly worse than $(s_1,\{S^*(x):x\leqslant s_1\})$ policy.
\end{theorem}

\begin{remark}
\label{rem:thm3}
In both cases, although $(s_1,S_1)$ policy and $(s_1,\{S^*(x):x\leqslant s_1\})$ policy take different actions at $x=\underline{s}$, the definition of $\underline{s}$ in \eqref{eq:underline-s} implies that the discounted costs under these two policies are the same.
\end{remark}

The above theorem tells us that if $A_1^*>A_2^*$, there must exist some $x\in(-\infty,S_1-Q)$ such that $(s_1,S_1)$ policy is not optimal. At these initial levels,  Theorem \ref{thm3} also has shown that $(s_1,\{S^*(x):x\leqslant s_1\})$ policy performs better than  $(s_1,S_1)$ policy. So whether $(s_1,\{S^*(x):x\leqslant s_1\})$  is an optimal policy when $(s_1,S_1)$ is not optimal (i.e., when $x<S_1-Q$)?  In the following theorem, we show that
under some conditions, $(s_1,\{S^*(x):x\leqslant s_1\})$ policy is optimal in a large subset of admissible policies.
The proof of this theorem can be found in Section \ref{sec:proof-thm4}.
Before presenting the theorem, we introduce the concept of the local time at some point $a$  for semimartingales.
For the semimartingale $Z$ given by (\ref{eq:Z}), since $\abs{x-a}$ is convex, one has
\begin{align}
\label{eq:A}
\abs{Z(t)-a}=\abs{Z(0)-a}+\int_{0+}^t \operatorname{sign}(Z(s-)-a)\,\mathrm{d}Z(s)+A_t^a,
\end{align}
where $\operatorname{sign}(x)$ is the sign function with $\operatorname{sign}(x)=\mathbbm{1}_{\{z>0 \}}-\mathbbm{1}_{\{z\leqslant 0\}}$, and $A_t^{a}$ is an adapted, right continuous, increasing process; see Theorem 66 of \S IV in \cite{Protter2005}.

\begin{definition}[\cite{Protter2005}]
\label{def:localtime}
Let $Z$ be the semimartingale defined by {\rm (\ref{eq:Z})}, and let $A^a$ be defined in \eqref{eq:A}. The local time at $a$ of $Z$, denoted by $L_t^a(Z)$, is defined to be the process given by
\begin{align*}
L_{t}^{a}(Z)=A_t^{a}-\sum_{0<s \leq t}\Big(\abs{Z(s)-a}-\abs{Z(s-)-a}-\Delta Z(s)(\mathbbm{1}_{\{Z(s-)>a \}}-\mathbbm{1}_{\{Z(s-)\leqslant a\}})\Big).
\end{align*}
\end{definition}

\begin{theorem}
\label{thm4}
If $A_1^*>A_2^*$ and $\Xi(\underline{s}):=\mu k+g(\underline{s})-\beta (v_1(\bar{S})+K_2+k\cdot(\bar{S}-\underline{s}))\geqslant0$, then $(s_1,\{S^*(x):x\leqslant s_1\})$ policy defined in Definition {\rm\ref{def-policy}} is optimal  in $\mathcal{P}_{\{\underline{s}\}}$ for any initial level $x\in(-\infty,S_1-Q)$,
where
\begin{equation}
\label{eq:P}
\mathcal{P}_{\{\underline{s}\}}=\Big\{\phi\in\mathcal{P}: \mathbb{E}_x\Big[\int_0^{\infty} e^{-\beta t}\,\mathrm{d}L_t^{\underline{s}}(Z_\phi)\Big]=0\Big\}
\end{equation}
with that $Z_\phi$ is the inventory level process under policy $\phi$.
Further, the discounted cost for $x\in(-\infty,S_1-Q)$ under $(s_1,\{S^*(x):x\leqslant s_1\})$ policy is given by
\eqref{eq:valfun}.
\end{theorem}

The assumption $\Xi(\underline{s})\geqslant0$ is a technical one,
and we will see in the proof of Theorem \ref{thm4} that this assumption is used to ensure the cost function $\mathsf{DC}(x,(s_1,\{S^*(x):x\leqslant s_1\}))$ satisfies the condition \eqref{eq:lowerbound-1} in the lower bound theorem.
In fact, the condition $\Xi(\underline{s})\geqslant0$ holds in many cases, and in Section \ref{sec:condition-Q}
we will show that it always holds when the threshold $Q$ is large.

The set constraint \eqref{eq:P} is a mathematical condition for the admissible policies. We will see that a large class of admissible policies satisfy this condition. Let  $\mathcal{P}_{\{\underline{s}\}}'$ be a subset of policies that must place order up to a level higher than $\underline{s}$ once the inventory level $x$ is lower than or equal to $\underline{s}$. For any policy in $\mathcal{P}_{\{\underline{s}\}}'$, we see that the inventory level process $Z$ never pass $\underline{s}$ continuously, and thus \eqref{eq:P} holds, i.e., $\mathcal{P}_{\{\underline{s}\}}'\subset \mathcal{P}_{\{\underline{s}\}}$.

Recalling that the actions for $x\geqslant S_1-Q$ under both  $(s_1,\{S^*(x):x\leqslant s_1\})$ policy and $(s_1,S_1)$ policy are the same, thus we can integrate Theorems \ref{thm2} and \ref{thm4} to get the following result immediately.
\begin{corollary}
If $A_1^*>A_2^*$,  $(s_1,\{S^*(x):x\leqslant s_1\})$ policy is optimal in $\mathcal{P}$  for all $x\in[S_1-Q,\infty)$ while it is optimal in $\mathcal{P}_{\{\underline{s}\}}$ for $x\in(-\infty,S_1-Q)$ when $\Xi(\underline{s})\geqslant0$.
\end{corollary}

We will prove Theorem \ref{thm3} directly, and use two steps to prove Theorems \ref{thm1}, \ref{thm2}, and \ref{thm4} as follows.
First, we provide a comparison theorem under discounted cost version and then establish a generalized lower bound theorem (also called verification theorem); see Section \ref{sec:LBT}.
Second, we prove the optimality of the selected policies in these three theorems by checking associated cost functions satisfy the lower bound theorem; see Section \ref{sec:proof}.

\section{Generalized lower bound theorem}
\label{sec:LBT}

In this section, we establish a generalized lower bound theorem (Proposition \ref{prop:lowerbound}), which generalizes the lower bound theorem in literature (see e.g., \cite{DaiYao2013b},  \cite{HarrisonSellkeTaylor1983}, and \cite{OrmeciDaiVandeVate2008})  that requires continuously differentiable functions. Our lower bound theorem could apply to the functions that even do not have continuous first-order derivative at finite points.

To establish the lower bound theorem, we need a discounted cost version comparison theorem, which can be found in \cite{Jia2016} (see  an average cost version in \cite{HeYaoZhang2017}).

\begin{proposition}[Comparison Theorem]
\label{prop:comparison}
For any admissible policy $\phi\in\mathcal{P}$, there exists a sequence of admissible policies
\[
\phi_m\in\mathcal{P}^m=\{\phi\in\mathcal{P}: Z(t)\leqslant m \text{ at any ordering time $t$}\},\quad m=1,2,\cdots,
\]
such that $\lim_{m\to\infty} \mathsf{DC}(x,\phi_m)\leqslant  \mathsf{DC}(x,\phi)$.
\end{proposition}

By Comparison Theorem, it suffices to search an optimal policy in $\bar{\mathcal{P}}$, where
\[
\bar{\mathcal{P}}=\cup_{m=1}^{\infty}\mathcal{P}^m.
\]
Now, we are ready to show a generalized lower bound theorem, whose proof is shown at the end of this section.

\begin{proposition}[Generalized Lower Bound Theorem]
\label{prop:lowerbound}
 Let $f:\mathbb{R}\to\mathbb{R}$ be a continuous function with $f'$ and $f''$ continuous except at a finite set $\mathcal{S}=\{z_1,z_2,\ldots,z_{I}\}$, and the limits $f'(z_i\pm)\triangleq\lim_{x\rightarrow z_i\pm}f'(x)$,  $f''(z_i\pm)\triangleq\lim_{x\rightarrow z_i\pm}f''(x)$ for $z_i\in \mathcal{S}$ exist and are finite.  Assume that
\begin{equation}
\label{eq:lowerbound-1}
\Gamma f(x)-\beta f(x)+g(x)\geqslant 0 \quad \text{for any $x\in\mathbb{R}\setminus\mathcal{S}$}.
\end{equation}
Further assume that
\begin{equation}
\label{eq:lowerbound-2}
f(x_2)-f(x_1)\geqslant -K(x_2-x_1)-k\cdot(x_2-x_1) \quad \text{for any $x_1<x_2$},
\end{equation}
and that there exist positive constants $a_0$, $a_1$ and a positive integer $n$ such that
\begin{equation}
\label{eq:lowerbound-3}
\abs{f'(x)}<a_0+a_1 x^n,\quad \text{$\forall$ $x\in[0,\infty)\setminus \mathcal{S}$},
\end{equation}
and
\begin{equation}
\label{eq:lowerbound-4}
\abs{f'(x)}<a_0, \quad \text{$\forall$ $x\in(-\infty,0)\setminus \mathcal{S}$}.
\end{equation}
Then, $\mathsf{DC}(x,\phi)\geqslant f(x)$ for any initial inventory level $x\in\mathbb{R}$ and any admissible policy
\begin{align*}
\phi\in\mathcal{P}^f=\{\phi\in\mathcal{P}: \mathbb{E}_x\Big[\int_0^{\infty} e^{-\beta t} \,\mathrm{d} L_t^{z_i}(Z_\phi)\Big]=0, \forall z_i\in\mathcal{S}_f\},
\end{align*}
where $\mathcal{S}_f=\{z_i\in\mathcal{S}:f'(z_i+)-f'(z_i-)<0\}$, and $Z_\phi$ is the inventory level process with policy $\phi$.
\end{proposition}

If $f'$ is continuous in $\mathbb{R}$ or $f'(z_i+)-f'(z_i-)\geqslant0$ for any $z_i\in\mathcal{S}$, the set $\mathcal{S}_f$ would be empty, and then
\[
\mathcal{P}^f=\mathcal{P}.
\]
Thus, we have the following corollary.
\begin{corollary}
\label{cor:lowerbound}
In Proposition {\rm \ref{prop:lowerbound}}, if we further assume that $f'$ is continuous or $f'(z_i+)-f'(z_i-)\geqslant0$ for any $z_i\in\mathcal{S}$, then $\mathsf{DC}(x,\phi)\geqslant f(x)$ for any initial inventory level $x\in\mathbb{R}$ and any admissible policy $\phi\in\mathcal{P}$.
\end{corollary}
Since Theorems \ref{thm1} and Theorem \ref{thm2} show the optimality of selected policies in whole $\mathcal{P}$, their proofs need Corollary \ref{cor:lowerbound}; while Theorem \ref{thm4} shows the optimality in a subset, its proof needs Proposition \ref{prop:lowerbound}.
At the end of this section, we give the proof of Proposition \ref{prop:lowerbound}.

\begin{proof}[Proof of Proposition \ref{prop:lowerbound}]
By Proposition \ref{prop:comparison}, it suffices to consider the policies in $\bar{\mathcal{P}}$ with
$\bar{\mathcal{P}}=\cup_{m=1}^{\infty}\mathcal{P}^m$. When the inventory process $Z$ is given by \eqref{eq:Z} with $\phi\in\bar{\mathcal{P}}$, it follows from Lemma 3 of \cite{HeYaoZhang2017} that function $f$ with properties \eqref{eq:lowerbound-3}-\eqref{eq:lowerbound-4} satisfies
\begin{align}
&\mathbb{E}_x\big[\abs{f(Z(t))}\big]<\infty\quad \forall t\geqslant 0,\notag\\
& \mathbb{E}_x\Big[\int_0^t f'(Z(s))^2\,\mathrm{d}s\Big]<\infty\quad \forall t\geqslant 0,\quad\text{and}\label{eq:E<infty}\\
&\lim_{t\to\infty}\frac{1}{t}\mathbb{E}_x\big[\abs{f(Z(t))\cdot \mathbbm{1}_{[0,\infty)}(Z(t))}\big]=0.\label{eq:lim=0}
\end{align}
It follows from Problem 6.24 of \S3.6 in \cite{KaratzasShreve1991} that,
if function $f:\mathbb{R}\to\mathbb{R}$ is continuous and $f'$ and $f''$ exist and are continuous except at a finite set $\mathcal{S}=\{z_1,z_2,\ldots,z_{I}\}$, and the limits $\lim_{x\rightarrow z_i\pm}f'(x)$ and  $\lim_{x\rightarrow z_i\pm}f''(x)$, $i=1,\ldots, I$, exist and are finite,
then function $f$ can be written as the difference of two convex functions.
Since $Z$ is a semimartingale, by the It\^{o} formula (see e.g., Theorem 70 of Section IV in \cite{Protter2005} and Proposition 4.12 of  \cite{Harrison2013}), we have
\begin{align}
\label{eq:lci}
f(Z(t))=&f(Z(0))+\sum_{0<s\leqslant t}(f(Z(s))-f(Z(s-))-f'(Z(s-)\Delta Z(s))\\
&+\int_{0}^{t}f'(Z(s-))\,\mathrm{d}Z(s)+\frac{1}{2}\int_{-\infty}^{+\infty}L_{t}^{z}(Z)\nu(\mathrm{d}z),\notag
\end{align}
where $f'$ is the left derivative of $f$, 
and $\nu$ is the signed measure (when restricted to compacts) defined by
$\nu(\mathrm{d}z)=f''(z)\,\mathrm{d}z$ for $z\in\mathbb{R}\setminus \mathcal{S}$
and $\nu\{z\}=f'(z+)-f'(z-)$ for $z\in\mathcal{S}$.
Thus, it follows from Corollary 9.47 in \cite{HeWangYan1992} that the last term in \eqref{eq:lci} can be calculated as
\begin{align*}
\frac{1}{2}\int_{-\infty}^{+\infty}L_{t}^{z}(Z)\nu(\,\mathrm{d}z)
=&\frac{1}{2}\Big[\int_{-\infty}^{+\infty}L_{t}^{z}(Z)f''(z)\,\mathrm{d}z
+\sum_{i=1}^{I}L_{t}^{z_{i}}(Z)(f'(z_{i}+)-f'(z_{i}-))\Big]\\
=&\frac{1}{2}\Big[\int_{0}^{t}f''(Z(s))\,\mathrm{d}[Z,Z]_{s}^{c}
+\sum_{i=1}^{I}L_{t}^{z_{i}}(Z)(f'(z_{i}+)-f'(z_{i}-))\Big]\\
=&\frac{1}{2}\Big[\int_{0}^{t}\sigma^2f''(Z(s))\,\mathrm{d}s
+\sum_{i=1}^{I}L_{t}^{z_{i}}(Z)(f' (z_{i}+)-f'(z_{i}-))\Big],
\end{align*}
and then \eqref{eq:lci} becomes
\begin{align*}
f(Z(t))
=&f(Z(0))+\int_{0}^{t}\Gamma f(Z(s))\,\mathrm{d}s+\int_{0}^{t}\sigma f'(Z(s))\,\mathrm{d}B(s)\\
&+\sum_{0<s\leqslant t}(f(Z(s))-f(Z(s-))+\frac{1}{2}\sum_{i=1}^{I}L_{t}^{z_{i}}(Z)
(f'(z_{i}+)-f'(z_{i}-)).
\end{align*}
Applying the formula of integration by parts, we have
\begin{align}
\label{eq:lt}
e^{-\beta t}f(Z(t))
=&f(Z(0))+\int_{0}^{t}e^{-\beta s}\,\mathrm{d}f(Z(s))-\beta\int_{0}^{t}e^{-\beta s}f(Z(s))\,\mathrm{d}s\\
=&f(Z(0-))+\int_{0}^{t}e^{-\beta s}\big[\Gamma f(Z(s))-\beta f(Z(s))\big]\,\mathrm{d}s\notag\\
&+\int_{0}^{t}\sigma e^{-\beta s}f'(Z(s))\,\mathrm{d}B(s)+\sum_{0\leqslant s\leqslant t}e^{-\beta s}(f(Z(s))-f(Z(s-))\notag\\
&+\frac{1}{2}\sum_{i=1}^{I}(f'(z_{i}+)-f'(z_{i}-))\int_{0}^{t}e^{-\beta s}\,\mathrm{d}L_{s}^{z_{i}}(Z).\notag
\end{align}
By \eqref{eq:E<infty}, we have
\[
\mathbb{E}_x\Big[\int_0^t \big(e^{-\beta s}f'(Z(s))\big)^2\,\mathrm{d}s\Big]\leqslant \mathbb{E}_x\Big[\int_0^t f'(Z(s))^2\,\mathrm{d}s\Big]<\infty,\quad \forall t\geqslant 0,
\]
then it follows from Theorem 3.2.1 in \cite{Oksendal2003} that
$\mathbb{E}_{x}[\int_{0}^{t}e^{-\beta s}f'(Z(s))\,\mathrm{d}B(s)]=0$.
Thus, taking expectation of \eqref{eq:lt} and using inequalities \eqref{eq:lowerbound-1}-\eqref{eq:lowerbound-2}, we have
\begin{align}
\label{eq:inequality}
\mathbb{E}_{x}\big[e^{-\beta t}f(Z(t))\big]
\geqslant& f(x)-\mathbb{E}_{x}\Big[\int_{0}^{t}e^{-\beta s}g(Z(s))\,\mathrm{d}s+\sum_{i=0}^{N(t)}e^{-\beta \tau_{i}}(K(\xi_{i})+k\cdot\xi_{i})\Big]\\
&+\frac{1}{2}\sum_{i=1}^{I}(f'(z_{i}+)-f'(z_{i}-))\mathbb{E}_{x}\Big[\int_{0}^{t}e^{-\beta s}\,\mathrm{d}L_{s}^{z_{i}}(Z)\Big],\notag
\end{align}
where we used \eqref{eq:lowerbound-2} in the form of
\begin{align*}
f(Z(s))-f(Z(s-))&=0 \quad\text{if} \quad s\notin \{\tau_0,\tau_1,\cdots,\tau_{N(t)}\},\\
f(Z(s))-f(Z(s-))
&\geqslant-K(\xi_i)-k\cdot\xi_i \quad \text{if} \quad s=\tau_i\in \{\tau_0,\tau_1,\cdots,\tau_{N(t)}\}.
\end{align*}
Since $L_{t}^{z_{i}}(Z),i=1,\cdots,I$, is a continuous, increasing process, we have
$\mathbb{E}_{x}[\int_{0}^{t}e^{-\beta s}\,\mathrm{d}L_{s}^{z_{i}}(Z)]\geqslant 0$,
and then for any policy in $\mathcal{P}^f$,
\[
\frac{1}{2}\sum_{z_{i}\in \mathcal{S}}(f'(z_{i}+)-f'(z_{i}-))\mathbb{E}_{x}\Big[\int_{0}^{t}e^{-\beta s}\,\mathrm{d}L_{s}^{z_{i}}(Z)\Big]\geqslant 0,
\]
which, together with \eqref{eq:inequality}, implies
\[
\mathbb{E}_{x}[e^{-\beta t}f(Z(t))]\geqslant f(x)-\mathbb{E}_{x}\Big[\int_{0}^{t}e^{-\beta s}g(Z(s))\,\mathrm{d}s+\sum_{i=0}^{N(t)}e^{-\beta \tau_{i}}(K(\xi_{i})+k\cdot\xi_{i})\Big].
\]
Letting $t$ go to infinity, we get
\[
\liminf_{t\to \infty }\mathbb{E}_{x}[e^{-\beta t}f(Z(t))]+{\sf DC}(x,\phi)\geqslant f(x).
\]
If $\liminf_{t\to \infty }\mathbb{E}_{x}[e^{-\beta t}f(Z(t))]\leqslant 0$, we have ${\sf DC}(x,\phi)\geqslant f(x)$;
otherwise,
\[
\liminf_{t\to \infty }\mathbb{E}_{x}[e^{-\beta t}f(Z(t))]>c\quad \text{for a strictly positive constant $c$},
\]
thus by inequality \eqref{eq:lim=0}, we further have
$\liminf_{t\to \infty}\mathbb{E}_{x}[e^{-\beta t}|f(Z(t))\mathbbm{1}_{\{Z(t)< 0)\}}|]>c$.
Then, there exists a sufficiently large $t_c$ such that
\[
\mathbb{E}_{x}[e^{-\beta t}|f(Z(t))\mathbbm{1}_{\{Z(t)<0\}}|]\geqslant\frac{1}{2}c\quad\text{for $t>t_c$}.
\]
Furthermore, by \eqref{eq:lowerbound-4} and the continuity of $f$, there exists a constant $b_0$ such that $\abs{f(z)}<a_{0}\abs{z}+b_0$ for $z\in(-\infty,0)$, thus
\begin{equation}
\label{eq:EZ>}
\mathbb{E}_{x}[\abs{Z(t)}]\geqslant\frac{ce^{\beta t}-2b_{0}}{2a_{0}}\quad\text{for $t>t_c$}.
\end{equation}
Furthermore, it follows from Assumption \ref{ass:h} (A1) and (A3) that there exist constants $a_1>0$ and $b_1\in\mathbb{R}$ such that
\[
g(x)\geqslant a_1\abs{x}-b_1\quad \text{for $x\in\mathbb{R}$},
\]
which yields
\begin{align*}
\mathbb{E}_{x}\Big[\int_{0}^{\infty}e^{-\beta t}g(Z(t))\,\mathrm{d}t\Big]
&=\int_{0}^{\infty}e^{-\beta t}\mathbb{E}_{x}[g(Z(t))]\,\mathrm{d}t\\
&\geqslant \int_{t_c}^{\infty}e^{-\beta t}\mathbb{E}_{x}[g(Z(t))]\,\mathrm{d}t\\
&\geqslant a_1 \int_{t_c}^{\infty}e^{-\beta t}\mathbb{E}_{x}[\abs{Z(t)}]\,\mathrm{d}t-\frac{b_1}{\beta}e^{-\beta t_c}\\
&=\infty,
\end{align*}
where the first equality holds by Tonelli's Theorem, and the last equality follows from \eqref{eq:EZ>}.
Therefore,  ${\sf DC}(x,\phi)=\infty$, and then ${\sf DC}(x,\phi)\geqslant f(x)$ holds for any $\phi\in \mathcal{P}^f$.
\end{proof}

\section{Proof of main results}
\label{sec:proof}

In this section, we prove our main results, Theorems \ref{thm1}, \ref{thm2}, \ref{thm3}, and \ref{thm4}.
Specifically, in Section \ref{sec:proof-thm1}, we prove Theorem \ref{thm1} by checking that the cost function $V_2(x)$ under $(s_2,S_2)$ policy satisfies all conditions in Corollary \ref{cor:lowerbound}; In Section \ref{sec:proof-thm2}, we prove Theorem \ref{thm2} by constructing a function $V_1(x)$ such that it satisfies all conditions in Corollary \ref{cor:lowerbound} as well as equals the cost under $(s_1,S_1)$ policy for $x\in[S_1-Q,\infty)$; In Section \ref{sec:proof-thm3}, we prove Theorem \ref{thm3} by comparing the cost under $(s_1,S_1)$ policy with the cost under $(s_1,\{S^*(x):x\leqslant s_1\})$ policy; In Section \ref{sec:proof-thm4}, we prove Theorem \ref{thm4} by checking that the cost function $\mathsf{DC}(x,(s_1,\{S^*(x):x\leqslant s_1\}))$ under $(s_1,\{S^*(x):x\leqslant s_1\})$ policy satisfies all conditions in Proposition \ref{prop:lowerbound}.

\subsection{The optimality of $(s_2,S_2)$ policy when $A_1^*\leqslant A_2^*$}
\label{sec:proof-thm1}

Before proving Theorem \ref{thm1}, we first give some useful properties.
\begin{lemma}\label{lem:A2}
When $A_1^*\leqslant A_2^*$, we must have
\begin{equation}
\label{eq:S-s}
S_1-s_1=Q\quad \text{and}\quad S_2-s_2>Q.
\end{equation}
\end{lemma}
\begin{proof}
We first prove the first part of  \eqref{eq:S-s}, i.e., when $A_1^*\leqslant A_2^*$, we have $S_1-s_1=Q$.
Recall that $(s_1,S_1)$ is the unique maximizer of $\mathcal{OP}_1: \max_{0<S-s\leqslant Q}A_1(s,S)$, thus we must have $S_1-s_1\leqslant Q$.
Suppose $S_1-s_1<Q$, if we can claim that
\begin{equation}
\label{eq:suppose-1}
\text{$(s_1,S_1)$ is a maximizer of }\max_{S-s>0}A_1(s,S),
\end{equation}
then
\begin{align*}
A_1^*&=A_1(s_1,S_1)=\max_{S-s>0}A_1(s,S)
>\max_{S-s>0}A_2(s,S)\geqslant \max_{S-s\geqslant Q}A_2(s,S)=A_2(s_2,S_2)=A_2^*,
\end{align*}
where the first inequality follows from the definition of $A_i$ in \eqref{eq:A(s,S)} and $K_2>K_1$. This contradicts with $A_2^*\geqslant A_1^*$. Thus,  $S_1-s_1=Q$ holds when $A_2^*\geqslant A_1^*$.

Next we prove that \eqref{eq:suppose-1} holds under condition $S_1-s_1<Q$.
If \eqref{eq:suppose-1} is not true, there must exist a maximizer $(\tilde{s}_1,\tilde{S}_1)$ with $Q<\tilde{S}_1-\tilde{s}_1<\infty$ (the finiteness of $\tilde{s}$ and $\tilde{S}$ can be proved in a similar way used in the proof of Lemma \ref{lem:sS-A} (c)) such that $A_1(\tilde{s}_1,\tilde{S}_1)>A_1(s_1,S_1)$ and
\begin{equation}
\label{eq:tilde-A}
\frac{\partial A_1}{\partial s}(\tilde{s}_1,\tilde{S}_1)=\frac{\partial A_1}{\partial S}(\tilde{s}_1,\tilde{S}_1)=0.
\end{equation}
Denote $\tilde{A}_1=A_1(\tilde{s}_1,\tilde{S}_1)$ and define
\[
\tilde{v}_1(x)=\frac{2}{\sigma^2}\frac{1}{\lambda_1+\lambda_2}
\Big[\int_x^{\infty} e^{\lambda_1(x-y)}g(y)\,\mathrm{d}y +\int_0^x e^{-\lambda_2(x-y)}g(y)\,\mathrm{d}y-\frac{1}{\lambda_2^2}\tilde{A}_1 e^{-\lambda_2x}
\Big],\ x\in\mathbb{R}.
\]
Then \eqref{eq:tilde-A} can be rewritten as
\begin{equation}
\label{eq:tilde-v'=-k}
\tilde{v}_1'(\tilde{s}_1)=\tilde{v}_1'(\tilde{S}_1)=-k,
\end{equation}
and we have
\begin{equation}
\label{eq:tilde-A>A2}
\tilde{A}_1=A_1(\tilde{s}_1,\tilde{S}_1)\geqslant A_1(s_2,S_2)>A_2(s_2,S_2)=A_2^*\geqslant A_1^*,
\end{equation}
where the second inequality follows from that the definition of $A_i$ in \eqref{eq:A(s,S)} and  $K_2>K_1$.
It follows from \eqref{eq:tilde-A>A2} and \eqref{eq:vi-derivate} that
$\tilde{v}_1'(x)>v_1'(x)$ for $x\in\mathbb{R}$,
which, together with \eqref{eq:tilde-v'=-k}, yields that
\begin{equation}
\label{eq:v1-tilde-v1}
v_1'(\tilde{s}_1)<\tilde{v}_1'(\tilde{s}_1)=-k\quad\text{and}\quad v_1'(\tilde{S}_1)<\tilde{v}_1'(\tilde{S}_1)=-k.
\end{equation}
Recall part ($d$) of Lemma \ref{lem:sS-A} that when $S_1-s_1<Q$, we have
$v_1'(s_1)=v_1'(S_1)=-k$,
which, together with \eqref{eq:v1-tilde-v1} and the strict quasi-convexity of $v_1'$,
implies that
\[
S_1-s_1>\tilde{S}_1-\tilde{s}_1>Q.
\]
This contradicts with the condition $S_1-s_1<Q$. Therefore,  when $S_1-s_1<Q$, $(s_1,S_1)$ must be a maximizer of optimization problem  $\max_{S-s>0}A_1(s,S)$.

It remains to prove the second part of \eqref{eq:S-s}, i.e., $S_2-s_2>Q$. Suppose $S_2-s_2=Q$, we have
\[
A_1^*=A_1(s_1,S_1)=\max_{0<S-s\leqslant Q}A_1(s,S)>A_1(s_2,S_2)>A_2(s_2,S_2)=A_2^*,
\]
where the first inequality follows from $S_2-s_2=Q$, and the second inequality follows from the definition of $A_i$ in \eqref{eq:A(s,S)} and $K_2>K_1$. This contradicts with $A_2^*\geqslant A_1^*$, and thus we have $S_2-s_2>Q$.
\end{proof}

Now we are ready to prove Theorem \ref{thm1}.
\begin{proof}[Proof of Theorem \ref{thm1}]
The function $V_2(x)$ defined in \eqref{eq:V2} is the discounted cost for initial level $x$ under $(s_2,S_2)$ policy. If we can show that $V_2$ satisfies the conditions in Corollary \ref{cor:lowerbound} (i.e., all conditions in Proposition \ref{prop:lowerbound} as well as continuity of the first derivative), then
\begin{equation*}
\mathsf{DC}(x,\phi)\geqslant V_2(x)\quad \text{for any $x\in\mathbb{R}$ and $\phi\in\mathcal{P}$},
\end{equation*}
i.e., $V_2(x)$ is the optimal cost for initial level $x$ and $(s_2,S_2)$ policy is an optimal policy.
What remains is to check that $V_2$ satisfies all conditions in Corollary \ref{cor:lowerbound}.

First, the definitions of $V_2$ and $v_2$ imply
\[
V_2(s_2)=v_2(s_2)=v_2(S_2)+K_2+k(S_2-s_2)=\lim_{x\uparrow s_2}V_2(x),
\]
then $V_2$ is continuous at $s_2$ and thus is continuous in $\mathbb{R}$.  Further, part $(e)$ in Lemma \ref{lem:sS-A} and $S_2-s_2>Q$ in \eqref{eq:S-s} imply
\[
\lim_{x\downarrow s_2}V_2'(x)=v_2'(s_2)=-k=\lim_{x\uparrow s_2}V_2'(x),
\]
thus, $V_2'$ is continuous in $\mathbb{R}$. In addition, $V_2''$ is continuous except  at  $x=s_2$,  $V_2''(s_2-)=\lim_{x\rightarrow s_2-}V_2''(x)=0$, and $V_2''(s_2+)=\lim_{x\rightarrow s_2+}V_2''(x)=v_2''(s_2)$.
Therefore, it remains to check that $V_2$ satisfies conditions \eqref{eq:lowerbound-1}-\eqref{eq:lowerbound-4}.

Check condition \eqref{eq:lowerbound-1}.
For $x\in[s_2,\infty)$, the definitions of $V_2$ and $v_2$ imply that
\begin{equation}
\label{eq:Gamma=0}
\Gamma V_2(x)-\beta V_2(x)+g(x)=\Gamma v_2(x)-\beta v_2(x)+g(x)=0.
\end{equation}
For $x\in(-\infty,s_2)$, we have
\[
\Gamma V_2(x)-\beta V_2(x)+g(x)=\mu k-\beta\big(v_2(S_2)+K_2+k(S_2-x)\big)+g(x).
\]
Let $\Phi_1(x)=\mu k-\beta\big(v_2(S_2)+K_2+k(S_2-x)\big)+g(x)$, then for $x\in(-\infty,s_2)$,
\[
\Phi_1'(x)=\beta k+g'(x)<0,
\]
where the inequality follows from Lemma \ref{lem:sS-A} $(f)$ by noting $s_2<x_2^*$ and $S_2-s_2>Q$ (see \eqref{eq:S-s}).
Thus, we have that for $x\in(-\infty,s_2)$,
\begin{equation}
\label{eq:Gamma>Phi}
\Gamma V_2(x)-\beta V_2(x)+g(x)=\Phi_1(x)>\Phi_1(s_2).
\end{equation}
Note that taking $x=s_2$ in \eqref{eq:Gamma=0}, we have
\begin{equation}
\label{eq:0=Gamma}
0=\Gamma V_2(s_2)-\beta V_2(s_2)+g(s_2)=\frac{1}{2}v_2''(s_2)-\mu v_2'(s_2)-\beta v_2(s_2)+g(s_2)
=\frac{1}{2}v_2''(s_2)+\Phi_1(s_2).
\end{equation}
Furthermore,  the strict quasi-convexity of $v_2'$ (see Remark \ref{rem:vi} (a)) and $s_2<x_2^*$ implies
$v_2''(s_2)<0$,
which, together with \eqref{eq:Gamma>Phi}-\eqref{eq:0=Gamma}, yields
\[
\Gamma V_2(x)-\beta V_2(x)+g(x)>0.
\]
The verification of condition \eqref{eq:lowerbound-1} is completed.

We check condition \eqref{eq:lowerbound-2} in three cases: $x_1<x_2<s_2$, $s_2\leqslant x_1<x_2$, and $x_1<s_2\leqslant x_2$.

\noindent$\mathsf{Case\ 1:}$ If $x_1<x_2< s_2$, we have
\[
V_2(x_2)-V_2(x_1)=-k\cdot(x_2-x_1)>-K(x_2-x_1)-k\cdot(x_2-x_1),
\]
thus \eqref{eq:lowerbound-2} holds.

\noindent$\mathsf{Case\ 2:}$ If $s_2\leqslant x_1<x_2$, we have
\begin{equation}
\label{eq:bar-v2-minus}
V_2(x_2)-V_2(x_1)=v_2(x_2)-v_2(x_1)=\int_{x_1}^{x_2}[v_2'(y)+k]\,\mathrm{d}y-k(x_2-x_1).
\end{equation}
When $x_2-x_1>Q$, we have
\begin{align}
\label{eq:int-v2}
\int_{x_1}^{x_2}[v_2'(y)+k]\,\mathrm{d}y
\geqslant \int_{s_2}^{S_2} [v_2'(y)+k]\,\mathrm{d}y
=v_2(S_2)-v_2(s_2)+k(S_2-s_2)
=-K_2
=-K(x_2-x_1),
\end{align}
where the inequality follows from the quasi-convexity of $v_2'$ and $v_2'(s_2)=v_2'(S_2)=-k$ (see Lemma \ref{lem:sS-A} ($e$) and  \eqref{eq:S-s}), and the last equality follows from $x_2-x_1>Q$. Thus, \eqref{eq:bar-v2-minus}-\eqref{eq:int-v2} imply that when $x_2-x_1>Q$,
\begin{equation}
\label{eq:bar-v-geq1}
V_2(x_2)-V_2(x_1)\geqslant -K(x_2-x_1)-k(x_2-x_1).
\end{equation}
When $x_2-x_1\leqslant Q$,  it follows from \eqref{eq:bar-v2-minus} that
\begin{equation}
\label{eq:bar-v1-minus}
V_2(x_2)-V_2(x_1)=\int_{x_1}^{x_2}[v_2'(y)+k]\,\mathrm{d}y-k(x_2-x_1)
\geqslant \int_{x_1}^{x_2}[v_1'(y)+k]\,\mathrm{d}y-k(x_2-x_1),
\end{equation}
where the inequality follows from the definition of $v_i'$  in \eqref{eq:vi-derivate} and $A_1^*\leqslant A_2^*$. In this case, if we can further prove
\begin{equation}
\label{eq:appendix-proof2}
\int_{x_1}^{x_2}[v_1'(y)+k]\,\mathrm{d}y\geqslant  \int_{s_1}^{S_1}[v_1'(y)+k]\,\mathrm{d}y=-K_1,
\end{equation}
then \eqref{eq:bar-v1-minus}-\eqref{eq:appendix-proof2} imply
\begin{equation}
\label{eq:bar-v-geq2}
V_2(x_2)-V_2(x_1)\geqslant -K_1-k(x_2-x_1)=-K(x_2-x_1)-k(x_2-x_1),
\end{equation}
where the equality follows from $x_2-x_1\leqslant Q$. Therefore, it follows from \eqref{eq:bar-v-geq1} and \eqref{eq:bar-v-geq2} that when $x_1<x_2<s_2$, condition \eqref{eq:lowerbound-2} holds.

We next prove \eqref{eq:appendix-proof2}.
Since $0<x_2-x_1\leqslant Q$, it follows from the strict quasi-convexity of $v_1'$ that there exist $\alpha_1$ and $\alpha_2$ with $\alpha_1<\alpha_2$ such that
\begin{equation}
\label{eq:alpha12}
\alpha_2-\alpha_1=x_2-x_1\quad \text{and}\quad v_1'(\alpha_1)=v_1'(\alpha_2),
\end{equation}
and then
\begin{equation}
\label{eq:v1-alpha}
v_1'(x)
\begin{cases}
<v_1'(\alpha_1)& \text{for $x\in(\alpha_1,\alpha_2)$},\\
>v_1'(\alpha_1)& \text{for $x\in[\alpha_1,\alpha_2]^c$},
\end{cases}
\end{equation}
where $A^c$ denotes the complementary of set $A$ with respect to $\mathbb{R}$. Therefore,
\begin{align}
\label{eq:int>int-1}
\int_{x_1}^{x_2}[v_1'(y)+k]\,\mathrm{d}y
&=\int_{[x_{1},x_{2}]\cap[\alpha_1,\alpha_2]}[v_{1}'(y)+k]\,\mathrm{d}y
+\int_{[x_{1},x_{2}]\cap[\alpha_1,\alpha_2]^c}[v_{1}'(y)+k]\,\mathrm{d}y\\
&\geqslant \int_{[x_{1},x_{2}]\cap[\alpha_1,\alpha_2]}[v_{1}'(y)+k]\,\mathrm{d}y
+[v_{1}'(\alpha_1)+k]\nu([x_{1},x_{2}]\cap[\alpha_1,\alpha_2]^c)\nonumber\\\
&= \int_{[x_{1},x_{2}]\cap[\alpha_1,\alpha_2]}[v_{1}'(y)+k]\,\mathrm{d}y+[v_{1}'(\alpha_1)+k]\nu([\alpha_1,\alpha_2]\cap[x_{1},x_{2}]^c)\nonumber\\\
&\geqslant\int_{[\alpha_1,\alpha_2]\cap[x_{1},x_{2}]}[v_{1}'(y)+k]\,\mathrm{d}y
+\int_{[\alpha_1,\alpha_2]\cap[x_{1},x_{2}]^c}[v_{1}'(y)+k]\,\mathrm{d}y\nonumber\\\
&= \int_{\alpha_{1}}^{\alpha_{2}}[v_{1}'(y)+k]\,\mathrm{d}y,\nonumber
\end{align}
where both inequalities follow from \eqref{eq:v1-alpha} and $\nu(A)$ denotes the Lebesgue measure of set $A$, the second equality follows from $\nu([x_{1},x_{2}]\cap[\alpha_1,\alpha_2]^c)=\nu([\alpha_1,\alpha_2]\cap[x_{1},x_{2}]^c)$ (due to $\alpha_2-\alpha_1=x_2-x_1$).
Recall that $S_1-s_1=Q$ (see \eqref{eq:S-s}), it follows from the strict quasi-convexity of $v_1'$, the second part of \eqref{eq:alpha12}, and $v_1'(s_1)=v_1'(S_1)$ (Lemma \ref{lem:sS-A} ($d$)) that
\[
s_1\leqslant \alpha_1\leqslant\alpha_2\leqslant S_2\quad \text{and}\quad
v_1'(\alpha_1)=v_1'(\alpha_2)\leqslant v_1'(s_1)=v_1'(S_1)\leqslant -k,
\]
where the last inequality in the second part is from Lemma \ref{lem:sS-A} ($d$) for the case when $S_1-s_1=Q$.
Then,
\begin{align}
\int_{\alpha_{1}}^{\alpha_{2}}[v_{1}'(y)+k]\,\mathrm{d}y
\geqslant \int_{s_1}^{S_1}[v_{1}'(y)+k]\,\mathrm{d}y
=v_1(S_1)-v_1(s_1)-k\cdot(S_1-s_1)
=-K_1.\label{eq:int>int-2}
\end{align}
Hence, \eqref{eq:int>int-1}-\eqref{eq:int>int-2} imply \eqref{eq:appendix-proof2}.

\noindent$\mathsf{Case\ 3:}$ If $x_1<s_2\leqslant x_2$, we have
\begin{align*}
V_2(x_2)-V_2(x_1)
&=v_2(x_2)-v_2(s_2)+k\cdot(x_2-s_2)-k\cdot(x_2-x_1)\\
&\geqslant -K(x_2-s_2)-k\cdot(x_2-x_1)\\
&\geqslant -K(x_2-x_1)-k\cdot(x_2-x_1),
\end{align*}
where the first inequality holds because $v_2(x_2)-v_2(s_2)\geqslant-K(x_2-s_2)-k\cdot(x_2-s_2)$ holds from \eqref{eq:bar-v-geq1} and \eqref{eq:bar-v-geq2}, and the last inequality follows from that $K(\cdot)$ is an increasing function.

Finally, we check \eqref{eq:lowerbound-3}-\eqref{eq:lowerbound-4}.
It follows from the definition of $V_2$ in \eqref{eq:V2} and \eqref{eq:vi-derivate} that
\begin{align*}
V_2'(x)
&=
\begin{cases}
v_2'(x), & x\geqslant s_2,\\
-k, & x<s_2,
\end{cases}\\
&=
\begin{cases}
\frac{2}{\sigma^{2}} \frac{1}{\lambda_{1}+\lambda_{2}}\Big[\int_x^{\infty} e^{\lambda_1(x-y)}g'(y)\,\mathrm{d}y+\int_0^x e^{-\lambda_2(x-y)}g'(y)\,\mathrm{d}y+\frac{1}{\lambda_2} A_2^*e^{-\lambda_2 x}\Big], & x\geqslant s_2,\\
-k, & x<s_2,
\end{cases}
\end{align*}
where the expression of $v_2'(x)$ is obtained by integration by parts.
Since $g'$ is polynomially bounded (see Remark \ref{rem:h}), we can easily obtain \eqref{eq:lowerbound-3}-\eqref{eq:lowerbound-4}.
\end{proof}

\subsection{The optimality of $(s_1,S_1)$ policy for $x\in[S_1-Q,\infty)$ when $A_1^*>A_2^*$}
\label{sec:proof-thm2}

In this section, we aim to prove Theorem \ref{thm2}.
If we can construct a function, embodied by $V_1$, satisfies all conditions in Corollary \ref{cor:lowerbound}, then
\[
\mathsf{DC}(x,\phi)\geqslant V_1(x)\quad \text{for any $x\in\mathbb{R}$ and any $\phi\in\mathcal{P}$}.
\]
If $V_1$ further satisfies the following properties
\begin{equation}
\label{eq:V1-con2}
V_1(x)=\mathsf{DC}(x, (s_1,S_1)) \quad \text{if and only if $x\in[S_1-Q,\infty)$},
\end{equation}
then $\mathsf{DC}(x,\phi)\geqslant \mathsf{DC}(x, (s_1,S_1))$ for any $x\in[S_1-Q,\infty)$ and any $\phi\in\mathcal{P}$,
i.e., $(s_1,S_1)$ policy is optimal in  $\mathcal{P}$ for any initial level $x\in[S_1-Q,\infty)$.
Therefore, what remains is to construct the function $V_1$, and then to prove that it satisfies \eqref{eq:V1-con2} and all conditions in Corollary \ref{cor:lowerbound}.

First, we construct $V_1$. It follows from \eqref{eq:vi-derivate} and integration by parts that
\begin{equation}
\label{eq:v1'}
v_{1}'(x)=\frac{2}{\sigma^{2}} \frac{1}{\lambda_{1}+\lambda_{2}}\Big[\int_x^{\infty} e^{\lambda_1(x-y)}g'(y)\,\mathrm{d}y+\int_0^x e^{-\lambda_2(x-y)}g'(y)\,\mathrm{d}y+\frac{1}{\lambda_2} A_1^*e^{-\lambda_2 x}\Big],
\end{equation}
which implies that
\begin{align*}
\lim_{x\to-\infty}\frac{v_1'(x)}{e^{-\lambda_2 x}}
&=\frac{2}{\sigma^{2}} \frac{1}{\lambda_{1}+\lambda_{2}}\Big[ \lim_{x\to-\infty}\frac{\int_{x}^{\infty}e^{-\lambda_1 y}g'(y)\,\mathrm{d}y}{e^{-(\lambda_1+\lambda_2)x}}+ \frac{1}{\lambda_{2}}(A_1^*-\underline{A})\Big]\\
&=\frac{2}{\sigma^{2}} \frac{1}{\lambda_{1}+\lambda_{2}}\frac{1}{\lambda_{2}}
(A_1^*-\underline{A})\\
&>0,
\end{align*}
where the first equality follows from the definition of $\underline{A}$ in \eqref{eq:underline-A},  the second equality follows from
the polynomial boundedness of $g'$, and  the inequality follows from $A_1^*>\underline{A}$.
Therefore, we have
\begin{equation}
\label{eq:v'infty}
\lim_{x\to-\infty}v_1'(x)=\infty.
\end{equation}
It also follows from \eqref{eq:v1'} that
\begin{align}
\label{eq:lim-v1'}
\lim_{x\to\infty}v_{1}'(x)=\lim_{x\to\infty}\frac{g'(x)}{\beta}>0,
\end{align}
where the equality follows from the L'Hospital rule, and the inequality follows from the convexity of $g$ and $g'(x)>0$ for $x>0$.
Therefore, \eqref{eq:v'infty}-\eqref{eq:lim-v1'}, $v_1'(s_1)=v_1'(S_1)\leqslant -k$ (see Lemma \ref{lem:sS-A} ($d$)) and the quasi-convexity of $v_1'$ imply that
there exist unique $\bar{s}_1$ and $\bar{S}_1$ with $\bar{s}_1\leqslant s_1<S_1\leqslant \bar{S}_1$ such that
\begin{equation}
\label{eq:v1(bar-s)}
v_1'(\bar{s}_1)=v_1'(\bar{S}_1)=-k.
\end{equation}
Let
\begin{equation}
\label{eq:V1-def}
V_1(x)=
\begin{cases}
v_1(x), & \text{for $x\geqslant \bar{s}_1$},\\
v_1(\bar{s}_1)+k\cdot(\bar{s}_1-x), &\text{for $x<\bar{s}_1$}.
\end{cases}
\end{equation}

Next, we show that $V_1$ satisfies \eqref{eq:V1-con2}.  We prove this in two cases: $S_1-s_1<Q$ and $S_1-s_1=Q$.\\
\textsf{Case 1:} If $S_1-s_1<Q$, since $v_1'$ is strictly quasi-convex,
Lemma \ref{lem:sS-A} ($d$) and \eqref{eq:v1(bar-s)} imply that $s_1=\bar{s}_1$ and $S_1=\bar{S}_1$. Thus,
\begin{equation}
\label{eq:bar-V1-case1}
V_1(x)=
\begin{cases}
v_1(x), & \text{for $x\geqslant s_1$},\\
v_1(s_1)+k\cdot(s_1-x)=v_1(S_1)+K_1+k\cdot(S_1-x), &\text{for $x<s_1$},
\end{cases}
\end{equation}
where the equality in the case when $x<s_1$ follows from $v_1(s_1)=v_1(S_1)+K_1+k\cdot(S_1-s_1)$.
Furthermore, the discounted cost under $(s_1,S_1)$ policy is
\begin{equation}
\label{eq:DC-sS-1}
\mathsf{DC}(x, (s_1,S_1)) =
\begin{cases}
v_1(x), & \text{for $x\geqslant s_1$},\\
v_1(S_1)+K_1+k\cdot(S_1-x), &\text{for $S_1-Q\leqslant x<s_1$},\\
v_1(S_1)+K_2+k\cdot(S_1-x), &\text{for $x<S_1-Q$},
\end{cases}
\end{equation}
Comparing \eqref{eq:bar-V1-case1} with \eqref{eq:DC-sS-1}, we obtain that $V_1(x)=\mathsf{DC}(x, (s_1,S_1)) $ holds if and only if $x\in[S_1-Q,\infty)$.\\
\textsf{Case 2:} If $S_1-s_1=Q$, we have $v_1'(s_1)\leqslant -k$ and thus $\bar{s}_1\leqslant s_1$,
which implies that for $x\geqslant S_1-Q=s_1$
\[
V_1(x)=v_1(x)=\mathsf{DC}(x, (s_1,S_1)),
\]
and for $x<S_1-Q$,
\begin{align*}
V_1(x)&=v_1(s_1)+k\cdot(s_1-x)=v_1(S_1)+K_1+k\cdot(S_1-x)\\
&<v_1(S_1)+K_2+k\cdot(S_1-x)\\
&=\mathsf{DC}(x, (s_1,S_1)).
\end{align*}
Therefore, $V_1(x)=\mathsf{DC}(x, (s_1,S_1)) $ holds if and only if $x\in[S_1-Q,\infty)$.

Finally, we need to check that $V_1$ satisfies all conditions in Corollary \ref{cor:lowerbound}.
The proof of this part is very similar to that of Theorem \ref{thm1},
and thus it is put into the appendix; see Appendix \ref{app:supplement-thm2}.

\subsection{Proof of Theorem \ref{thm3}}
\label{sec:proof-thm3}

($a$) If $v_1'(s_1)=v_1'(S_1)=-k$, we need to show that for any $x\in(\underline{s},S_1-Q)$,
\begin{equation}
\label{eq:DC>DC}
\mathsf{DC}(x,(s_1,S_1))>\mathsf{DC}(x,(s_1,\{S^*(x):x\leqslant s_1\})).
\end{equation}
Recalling from \eqref{eq:S=S}, we have $\bar{S}=S_1$ in this case.  Then, we have that
for $x\in [\underline{s},S_1-Q)$,
\begin{align}
\label{eq:DC-}
&\mathsf{DC}(x,(s_1,S_1))-\mathsf{DC}(x,(s_1,\{S^*(x):x\leqslant s_1\})) \\
& \ \ \ =\big[v_1(S_1)+K_2+k\cdot(S_1-x)\big]-\big[v_1(x+Q)+K_1+kQ\big],\nonumber
\end{align}
where $\mathsf{DC}(x,(s_1,S_1))=v_1(S_1)+K_2+k\cdot(S_1-x)$ is due to $x<S_1-Q$, and
$\mathsf{DC}(x,(s_1,\{S^*(x):x\leqslant s_1\}))=v_1(x+Q)+K_1+kQ$ follows from \eqref{eq:underline-s<}-\eqref{eq:valfun}.
Denote $\Phi_2(x)=\big[v_1(S_1)+K_2+k\cdot(S_1-x)\big]-\big[v_1(x+Q)+K_1+kQ\big]$, we have
\begin{equation}
\label{eq:Psi}
\Phi_2(\underline{s})=0\quad \text{and}\quad \Phi_2'(x)=-v_1'(x+Q)-k>0\quad \text{for $x\in (\underline{s},S_1-Q)$},
\end{equation}
where $\Phi_2(\underline{s})=0$ is due to  the definition of $\underline{s}$ in \eqref{eq:underline-s} and $\bar{S}=S_1$, and the inequality follows from $v_1'(x+Q)<-k$ because $x+Q\in (\underline{s}+Q,S_1)\subset(s_1,S_1)$ (noting that $\underline{s}\geqslant s_1-Q$ from Lemma \ref{lem:bar-S} (b)).
Thus, \eqref{eq:DC-}-\eqref{eq:Psi} imply \eqref{eq:DC>DC}.

($b$) If $v_1'(s_1)=v_1'(S_1)<-k$,
when $x\in [\underline{s},S_1-Q)$ (this interval may be empty in this case), it is similar to \eqref{eq:Psi} that we also have
\begin{align*}
&\mathsf{DC}(x,(s_1,S_1))=\mathsf{DC}(x,(s_1,\{S^*(x):x\leqslant s_1\}))\quad \text{for $x=\underline{s}$},\\
&\mathsf{DC}(x,(s_1,S_1))>\mathsf{DC}(x,(s_1,\{S^*(x):x\leqslant s_1\}))\quad \text{for $x\in (\underline{s},S_1-Q)$}.
\end{align*}
When $x\in(-\infty,\underline{s})$, we have
\begin{align*}
\mathsf{DC}(x,(s_1,S_1))-\mathsf{DC}(x,(s_1,\{S^*(x):x\leqslant s_1\}))
&=\big[v_1(S_1)+K_2+k\cdot(S_1-x)\big]-\big[v_1(\bar{S})+K_2+k\cdot(\bar{S}-x)\big]\\
&=v_1(S_1)-v_1(\bar{S})+k\cdot(S_1-\bar{S})\\
&>0,
\end{align*}
where the inequality follows from $v_1'(x)+k<v_1'(\bar{S})+k=0$ for $x\in[S_1,\bar{S})$.
\qed

\subsection{The optimality of $(s_1,\{S^*(x):x\leqslant s_1\})$ policy for $x\in(-\infty,S_1-Q)$ when $A_1^*>A_2^*$}
\label{sec:proof-thm4}

In this section, we aim to prove Theorem \ref{thm4}.
Let
\begin{equation}
\label{eq:bar-V1}
\bar{V}_1(x)=\mathsf{DC}(x,(s_1,\{S^*(x):x\leqslant s_1\})),
\end{equation}
where $\mathsf{DC}(x,(s_1,\{S^*(x):x\leqslant s_1\}))$ is given in \eqref{eq:valfun}.
We will show that $\bar{V}_1(x)$ satisfies all the conditions specified in Proposition \ref{prop:lowerbound}.

First, we check the continuity of $\bar{V}_1$, $\bar{V}_1'$, and $\bar{V}_1''$, and prove $\mathcal{S}_{\bar{V}_1}=\{\underline{s}\}$ by considering two cases: $S_1-s_1<Q$ and $S_1-s_1=Q$. \\
\textsf{Case 1:} $S_1-s_1<Q$. It follows from the definition of $\bar{V}_1(x)$ and the definition of $\underline{s}$ in \eqref{eq:underline-s} that $\bar{V}_1(x)$ is continuous. Furthermore, it follows from $S_1-s_1<Q$ and Lemma \ref{lem:sS-A} ($d$) that
\begin{equation*}
v_1'(s_1)=v_1'(S_1)=-k,
\end{equation*}
which implies that $\bar{V}_1'$ is continuous at $s_1$ and $S_1-Q$, and thus it is continuous except at $\underline{s}$.
In addition,  $\bar{V}_1'(\underline{s}-)=\lim_{x\rightarrow \underline{s}-}\bar{V}_1'(x)=-k$ and $\bar{V}_1'(\underline{s}+)=\lim_{x\rightarrow \underline{s}+}\bar{V}_1'(x)=v_1'(\underline{s}+Q)$. Furthermore, $\bar{V}_1''$ is continuous except at $\{s_1,S_1-Q,\underline{s}\}$ and $\bar{V}_1''$ has finite left and right limits at these three points.
Note that $\bar{V}_1'$ is continuous except at $\underline{s}$, thus it is the unique possible point in $\mathcal{S}_{\bar{V}_1}$.
Further, it follows from \eqref{eq:valfun},  \eqref{eq:bar-V1} and the strict quasi-convexity of $v_1'$ that
\begin{equation}
\label{eq:bar-V1'}
\bar{V}_1'(\underline{s}+)=v_1'(\underline{s}+Q)< v_1'(s_1)=-k=\bar{V}_1'(\underline{s}-),
\end{equation}
which implies $\mathcal{S}_{\bar{V}_1}=\{\underline{s}\}$.\\
\textsf{Case 2:} $S_1-s_1=Q$. In this case, $(S_1-Q,s_1]$ is an empty set, and then
\[
\bar{V}_1(x)=
\begin{cases}
v_1(x), & \text{for $x\in(s_1,\infty)$},\\
v_1(x+Q)+K_1+kQ,  & \text{for $x\in[\underline{s},s_1]$},\\
v_1(\bar{S})+K_2+k(\bar{S}-x), & \text{for $x\in(-\infty,\underline{s})$}.
\end{cases}
\]
It follows from $v_1(s_1)=v_1(S_1)+K_1+k\cdot(S_1-s_1)=v_1(s_1+Q)+K_1+k Q$ and the definition of $\underline{s}$ that $\bar{V}_1$ is continuous in $\mathbb{R}$.
Furthermore, since $\bar{V}_1'(s_1+)=v_1'(s_1)=v_1'(S_1)=\bar{V}_1'(s_1-)$, we have that $\bar{V}_1'$ is continuous except at $\underline{s}$. In addition,
$\bar{V}_1'(\underline{s}-)=\lim_{x\rightarrow \underline{s}-}\bar{V}_1'(x)=-k$ and $\bar{V}_1'(\underline{s}+)=\lim_{x\rightarrow \underline{s}+}\bar{V}_1'(x)=v_1'(\underline{s}+Q)$. Furthermore, $\bar{V}_1''$ is continuous except at $\{s_1,\underline{s}\}$ and $\bar{V}_1''$ has finite left and right limits at these two points.
It is similar to \eqref{eq:bar-V1'} that we also have $\bar{V}_1'(\underline{s}+) <\bar{V}_1'(\underline{s}-)$, which implies $\mathcal{S}_{\bar{V}_1}=\{\underline{s}\}$.

It remains to check conditions \eqref{eq:lowerbound-1}-\eqref{eq:lowerbound-4}. The detailed proof is similar to that in Theorems \ref{thm1}-\ref{thm2}, and can be found in Appendix \ref{app:supplement-thm4}.
\qed

\section{The condition $\Xi(\underline{s})\geqslant0$ in Theorem \ref{thm4}}
\label{sec:condition-Q}

In this section, we analyze the condition $\Xi(\underline{s})\geqslant0$ in Theorem \ref{thm4} first
by some numerical studies and then prove that it always holds when the threshold $Q$ is large; see Example \ref{example:Q} and Propositions \ref{prop:condition-Q}.

Recalling the definition of $\underline{s}$ in \eqref{eq:underline-s}, we can then take $\Xi(\underline{s})$ as a function of $Q$.
\begin{example}
\label{example:Q}
In this example, we set $\mu=0.2$, $\sigma=0.6$, $\beta=0.01$, $k=0.85$, $K_{1}=4$, $K_{2}=7$. We set $g(x)=hx^+-px^-$ with $h=0.08$, $p=0.12$ in Table \ref{table:Q1}, and set $g(x)=\alpha x^2$ with $\alpha=0.01$ in Table \ref{table:Q2}.
It can be observed that when $Q$ is small (i.e., $Q=1,2$ in Table \ref{table:Q1} and $Q=1,2,3$ in Table \ref{table:Q2}), we get $A_1^*\leqslant A_2^*$, thus $(s_2,S_2)$ policy is optimal; When $Q$ becomes moderate (i.e., $Q=3$ in Table \ref{table:Q1} and $Q=4$ in Table \ref{table:Q2}), we have $A_1^*>A_2^*$ but $\Xi(\underline{s})<0$; When  $Q$ becomes large (i.e., $Q=4,\ldots,10$ in Table \ref{table:Q1} and $Q=5,\ldots,10$ in Table \ref{table:Q2}), we get $A_1^*>A_2^*$ and $\Xi(\underline{s})\geqslant0$. We also find that $\underline{s}$ is decreasing in $Q$ and  $\Xi(\underline{s})$ is strictly increasing in $Q$ when $Q$ becomes large.
In fact, the following proposition shows that there exists a $\underline{Q}$ such that
$\Xi(\underline{s})\geqslant 0$ for any $Q\in [\underline{Q},\infty)$ and $\Xi(\underline{s})$ is strictly increasing in $[\underline{Q},\infty)$.
\end{example}

\begin{table}[htbp]
 \caption{The value of $\Xi(\underline{s})$ varies with the threshold $Q$ when we set $g(x)=hx^+-px^-$ with $h=0.08$, $p=0.12$}{\label{table:Q1}}
 \centering
\resizebox{\textwidth}{!}{
 \begin{tabular}{cccccccccc}
  \toprule
$Q$&$s_1$& $S_1$& $A_1^{*}$& $s_2$&$S_2$& $A_2^*$& $\bar{S}$& $\underline{s}$&$\Xi(\underline{s})$\\
  \midrule
1&-1.3536&-0.3536&-0.0446&-4.4183&3.4704&-0.0201 &-&- &-\\
2&-1.7653&0.2347&-0.0251&-4.4183&3.4704&-0.0201 &-&- &-\\
3&-2.2178&0.7822&-0.0193&-4.4183&3.4704&-0.0201 &3.3042&-2.8178 &-0.0650\\
4&-2.6508&1.3492&-0.0170&-4.4183&3.4704&-0.0201 &2.8046&-4.3217 &0.1540\\
5&-3.0733&1.9267&-0.0160&-4.4183&3.4704&-0.0201 &2.5975&-5.5441 &0.3159\\
6&-3.4897&2.5103&-0.0157&-4.4183&3.4704&-0.0201 &2.5418&-6.6070 &0.4465\\
7&-3.5118&2.5416&-0.0157&-4.4183&3.4704&-0.0201 &2.5416&-7.6072 &0.5651\\
8&-3.5118&2.5416&-0.0157&-4.4639&3.5361&-0.0201 &2.5416&-8.6072 &0.6837\\
9&-3.5118&2.5416&-0.0157&-5.5255&3.4745&-0.0339 &2.5416&-9.6072 &0.8023\\
10&-3.5118&2.5416&-0.0157&-6.0177&3.9823&-0.0368 &2.5416&-10.6072 &0.9209\\
 \bottomrule
 \end{tabular}
}
\end{table}

\begin{table}[htbp]
 \caption{The value of $\Xi(\underline{s})$ varies with the threshold $Q$ when we set $g(x)=\alpha x^2$ with $\alpha=0.01$}{\label{table:Q2}}
 \centering
\resizebox{\textwidth}{!}{
 \begin{tabular}{cccccccccc}
  \toprule
$Q$&$s_1$& $S_1$& $A_1^{*}$& $s_2$&$S_2$& $A_2^*$& $\bar{S}$& $\underline{s}$&$\Xi(\underline{s})$\\
  \midrule
1&-3.7879&-2.7879&-0.0461&-6.4436&3.1352&-0.0194 &-&- &-\\
2&-3.2882&-1.2882&-0.0280&-6.4436&3.1352&-0.0194 &-&- &-\\
3&-3.4550&-0.4550&-0.0219&-6.4436&3.1352&-0.0194 &-&- &-\\
4&-3.7885&0.2115&-0.0190&-6.4436&3.1352&-0.0194 &3.0757&-5.4063 &-0.0197\\
5&-4.1887&0.8112&-0.0175&-6.4436&3.1352&-0.0194 &2.7504& -7.1070&0.2060\\
6&-4.6221&1.3779&-0.0166&-6.4436&3.1352&-0.0194 &2.5584&-8.5993 &0.4428\\
7&-5.0746&1.9254&-0.0162&-6.4436&3.1352&-0.0194 &2.4609&-9.8886 &0.6776\\
8&-5.5136&2.4345&-0.0160&-6.4436&3.1352&-0.0194 &2.4345&-10.9747 &0.8970\\
9&-5.5136&2.4345&-0.0160&-6.4436&3.1352&-0.0194 &2.4345&-11.9747 &1.1180\\
10&-5.5136&2.4345&-0.0160&-6.6377&3.3623&-0.0194 &2.4345&-12.9747 &1.3586\\
\bottomrule
\end{tabular}
}
\end{table}

\begin{proposition}
\label{prop:condition-Q}
There exists a $\underline{Q}\in(0,\infty)$ such that $\Xi(\underline{s})\geqslant0$ holds for any $Q\in[\underline{Q},\infty)$ and $\Xi(\underline{s})$ is strictly increasing in $[\underline{Q},\infty)$.
\end{proposition}
\begin{proof}[Proof] 
We consider the following unconstrained problem
\[
A_1^{\dagger}=\min_{s<S} A_1(s,S).
\]
Let $v_1^{\dagger}(x)=v_{A_1^{\dagger}}(x)$ for $x\in\mathbb{R}$.
It is very similar to Lemma \ref{lem:sS-A} that there exists unique finite $(s_1^{\dagger},S_1^{\dagger})$ such that $A_1^{\dagger}=A_1(s_1^{\dagger},S_1^{\dagger})$ and
\[
\frac{\mathrm{d} v_1^{\dagger}(s_1^{\dagger})}{\mathrm{d} x}
=\frac{\mathrm{d} v_1^{\dagger}(S_1^{\dagger})}{\mathrm{d} x}
=-k.
\]
Let $Q^{\dagger}=S_1^{\dagger}-s_1^{\dagger}$, then $(s_1^{\dagger},S_1^{\dagger})$ is a feasible solution of the constrained  problem $\min_{0<S-s\leqslant Q} A_1(s,S) $ for any $Q\geqslant Q^{\dagger}$,
and thus it is the optimal solution, i.e., for any $Q\geqslant Q^{\dagger}$,
\[
s_1=s_1^{\dagger}\quad  S_1=S_1^{\dagger},\quad\text{and}\quad v_1(x)=v_1^{\dagger}(x), \quad x\in\mathbb{R}.
\]
Therefore,
\begin{equation}
\label{eq:S1-constant}
\text{$S_1$ is a constant for any $Q\in[Q^{\dagger},\infty)$},
\end{equation}
and it follows from the definition of $\bar{S}$ in Lemma \ref{lem:bar-S} ($a$) and the strict quasi-convexity of $v_1$ that
\begin{equation}
\label{eq:bar-S=S1}
\bar{S}=S_1=S_1^{\dagger}, \quad \text{for any $Q\geqslant Q^{\dagger}$}.
\end{equation}
Recall that $\underline{s}$ is defined by \eqref{eq:underline-s},
which can be rewritten as
\begin{equation}
\label{eq:underline-s-Q}
v_1(\underline{s}+Q)+k\cdot(\underline{s}+Q)=v_1(S_1)+K_2+kS_1-K_1,
\end{equation}
where we used $\bar{S}=S_1$ in \eqref{eq:bar-S=S1}.
We have proven the existence of $\underline{s}$ in $[s_1-Q,S_1-Q]$, thus $\underline{s}+Q\in[s_1,S_1]$. Let
$\rho=\underline{s}+Q$.
In equation \eqref{eq:underline-s-Q}, the right-side is a constant for any $Q\in[Q^{\dagger},\infty)$, then the left-side is also a constant, i.e., $\rho=\underline{s}+Q$ is a fixed constant.
Thus, for any $Q\in[Q^{\dagger},\infty)$,
\begin{align*}
\Xi(\underline{s})
&=\mu k+g(\underline{s})-\beta \big(v_1(\bar{S})+K_2+k\cdot(\bar{S}-\underline{s})\big)\\
&=\mu k+g(\rho-Q)+\beta k(\rho-Q)-\beta \big(v_1(S_1)+K_2+k S_1\big),
\end{align*}
which, together with $g'(\rho-Q)+\beta k<0$ ($\rho-Q=\underline{s}\leqslant s_1<x_1^*$ and Lemma \ref{lem:sS-A} ($f$))  and \eqref{eq:S1-constant}, implies
that $\Xi(\underline{s})$ is strictly increasing in $Q\in[Q^{\dagger},\infty)$ and tends to $\infty$ as $Q\to\infty$, and then there exist a $\underline{Q}\in[Q^{\dagger},\infty)$ such that $\Xi(\underline{s})\geqslant0$ for $Q\in[\underline{Q},\infty)$.
\end{proof}





\section{Concluding remarks}
\label{sec:conclusion}

In this paper, we studied a continuous-review stochastic inventory model, in which the inventory process is modeled as a drifted Brownian motion, the holding and shortage cost function is convex, and a discontinuous quantity-dependent setup cost is incurred by each order.
The setup cost discontinuity makes the optimal policy more complicated under the discounted cost criterion. To completely characterize the optimal policy,
we first selected two ``good" $(s,S)$ policies, denoted by $(s_i,S_i)$ policy, $i=1,2$, which were obtained by solving two models. Namely,
 $(s_1,S_1)$ is determined by solving the Brownian inventory problem with the order quantity constraint bounded by $Q$ and constant setup cost $K_1$, and  $(s_2,S_2)$ is given by solving the Brownian inventory problem with the order quantity constraint larger than $Q$ and constant setup cost $K_2$.
We showed that $(s_2,S_2)$ policy is optimal for our problem when it is better than $(s_1,S_1)$ policy, otherwise, $(s_1,S_1)$ policy is optimal for initial level $x\in[S_1-Q,\infty)$ but is worse than the state-dependent policy, $(s_1,\{S^*(x):x\geqslant s_1\})$ for $x\in(-\infty,S_1-Q)$.
We further proved the optimality of $(s_1,\{S^*(x):x\geqslant s_1\})$ policy for $x\in(-\infty,S_1-Q)$ in a large class of admissible policies.
To prove the optimality of the selected policies, we established a generalized lower bound theorem to handle the discontinuity issue for the derivative of the cost function.

It is worth to mention that the convexity in Assumption \ref{ass:h} for holding/shortage cost function $g$ might be generalized to the weakly quasi-convex case (see \cite{YaoChaoWu2015}) or the case when the convexity holds only on $(-\infty,0)$ (see  \cite{BensoussanLiuSethi2005}). These generalizations would lead the characterization of optimal policy parameters to become more tedious.


One natural extension to our model is to consider a more general setup cost function with $N$ steps. The approach developed in this paper may still work but the analysis would be very tedious. In view of what we analyzed for the two-step case,  we conjecture that the similar results hold: some $(s,S)$ policies are optimal in some cases while they can be dominated by a more generalized state-dependent policy in other cases.
Further, the approach could apply to the inventory problems in which the demand processes are modeled as more general stochastic processes with continuous paths, e.g., the mean-reverting processes (cf.  \cite{CadenillasLaknerPinedo2010OR}) and general diffusions (cf.  \cite{HelmesStockbridgeZhu2015}), because the inventory level can be controlled into a desired region.
Finally, our approach provides a building block toward studying the models with jump processes, e.g., Poisson processes (cf. \cite{PresmanSethi2006}), compound Poisson plus diffusion processes (cf.  \cite{BenkheroufBensoussan2009} and \cite{BensoussanLiuSethi2005}), jump diffusions (cf. \cite{DavisGuoWu2010} and  \cite{Liu2018}), and general L\'{e}vy processes (cf. \cite{Yamazaki2017}).
For example, when the demand process is the mixture of compound Poisson processes and Brownian motions, even under a given policy, the inventory level at any ordering time could drop to any level like the initial level in this work.
However, it would become very difficult to obtain and compare the cost functions under different policies.

\section{Appendix}
\label{sec:appendix}

\subsection{Proof of Lemma \ref{lem:sS-A}}
\label{app:proof-lem-sS-A}
 ($a$) As the proof is slightly longer, we first outline its main procedures. To prove $\underline{A}<A_i^*<\bar A$,  we construct a finite $(s,S)$ policy for each $i$ such that its corresponding $A_i(s,S)$ defined in \eqref{eq:A(s,S)} is larger than $\underline{A}$, which, by the definition of $A_i^*$, implies $\underline{A}<A_i^*$. In order to show $A_i^*<\bar A$, for any finite $(s,S)$ policy, we decompose $A_i(s,S)$ defined in \eqref{eq:A(s,S)} into two parts, $\bar A$ and a function of $(s,S)$. Then  $A_i^*<\bar A$ is established by verifying the positivity of this function.

 Following the above outline, we now show that there exists a finite $(s,S)$ such that $A_i(s,S)>\underline{A}$, and thus $A_i^*>\underline{A}$.
For $s<S\leqslant 0$, using integration by parts, we can rewrite $A_i(s,S)$ defined in \eqref{eq:A(s,S)} as
\begin{align}
\label{eq:Ai(s,S)-re}
A_i(s,S)=
\frac{\lambda_2^2}{e^{-\lambda_2 S}-e^{-\lambda_2 s}}b_i(s,S)+\underline{A},
\end{align}
where
\begin{align}
\label{eq:b-i}
b_i(s,S)=a_i(s,S)+a_{\lambda_1}(s,S)-a_{\lambda_2}(s,S)
\end{align}
 with
\begin{align*}
a_i(s,S)=&
\frac{1}{2}\left(\sigma^2(\lambda_1+\lambda_2)
\left[K_i+k(S-s)+\frac{g(S)-g(s)}{\beta}\right]\right), \ i=1,2;\\
a_{\lambda_1}(s,S)=&\frac{1}{\lambda_1}\left(e^{\lambda_1S}\int_S^{\infty}e^{-\lambda_1y}g'(y)\,\mathrm{d}y
-e^{\lambda_1s}\int_s^{\infty}e^{-\lambda_1y}g'(y)\,\mathrm{d}y\right);
\\
a_{\lambda_2}(s,S)=&\frac{1}{\lambda_2}\left(e^{-\lambda_2S}\int^S_{-\infty}e^{\lambda_2y}g'(y)\,\mathrm{d}y
-e^{-\lambda_2s}\int^s_{-\infty}e^{\lambda_2y}g'(y)\,\mathrm{d}y\right),
\end{align*}
and $\underline{A}$ is defined in \eqref{eq:underline-A}.
Since $e^{-\lambda_2 S}-e^{-\lambda_2 s}<0$ for any $s<S$, we only need to prove that there exists a finite $(s,S)$ such that $b_i(s,S)<0$ for $i=1,2$.

First we consider $b_1(s,S)$. Since $S-s\in [0, Q]$, we just choose $S=s+Q$ and show that
\begin{align}
\label{eq:b1-0}
\mbox{there exists $s$ with $s<-Q$ such that} \ \ b_1(s,s+Q)<0.
\end{align}
Note that from the convexity of $g$, there are two possibilities for the limit of $g'(s)$ as $s\rightarrow -\infty$, namely,
\begin{align*}
\lim_{s\rightarrow -\infty}g'(s)>-\infty \ \ \mbox{or} \ \  \lim_{s\rightarrow -\infty}g'(s)=-\infty.
\end{align*}
In the following, the proof of \eqref{eq:b1-0} is given according to these two possible cases.

{\bf Case 1:} $\lim_{s\rightarrow -\infty}g'(s)>-\infty$. For the first term in (\ref{eq:b-i}) with $i=1$,
\begin{align}
\label{eq:a-1}
\lim_{s\rightarrow -\infty}a_1(s,s+Q)\leqslant \lim_{s\rightarrow -\infty}
\frac{\sigma^2(\lambda_1+\lambda_2)}{2}\left(K_1+kQ+g'(s+Q)\frac{Q}{\beta}\right)<0,
\end{align}
where the first inequality follows from the convexity of $g$, and the second one is from  $\lim_{s\rightarrow-\infty}[g'(s+Q)+\beta k]<-\beta K_1/Q$ given by Assumption \ref{ass:h} (A4).

For the second and third terms in (\ref{eq:b-i}),
\begin{align}
a_{\lambda_1}(s,s+Q)=&\frac{\int_{s+Q}^{\infty}e^{-\lambda_1y}g'(y)\,\mathrm{d}y
-e^{-\lambda Q}\int_s^{\infty}e^{-\lambda_1y}g'(y)\,\mathrm{d}y}{\lambda_1 e^{-\lambda_1(s+Q)}}\label{eq:a-lambda-1}\\
=&\frac{\int_{s+Q}^{\infty}e^{-\lambda_1y}g'(y)\,\mathrm{d}y
-\int_{s+Q}^{\infty}e^{-\lambda_1y}g'(y-Q)\,\mathrm{d}y}{\lambda_1 e^{-\lambda_1(s+Q)}};\nonumber\\
 a_{\lambda_2}(s,s+Q)=&\frac{\int^{s+Q}_{-\infty}e^{\lambda_2y}g'(y)\,\mathrm{d}y
-e^{\lambda_2Q}\int^s_{-\infty}e^{\lambda_2y}g'(y)\,\mathrm{d}y}{\lambda_2e^{\lambda_2(s+Q)}}\label{eq:a-lambda-2}\\
=&\frac{\int^{s+Q}_{-\infty}e^{\lambda_2y}g'(y)\,\mathrm{d}y -\int^{s+Q}_{-\infty}e^{\lambda_2y}g'(y-Q)\,\mathrm{d}y}{\lambda_2e^{\lambda_2(s+Q)}}.\nonumber
\end{align}
Let
\begin{align*}
c_1(s,s+Q)=&\frac{\int_{s+Q}^{\infty}e^{-\lambda_1y}g'(y)\,\mathrm{d}y}{\lambda_1 e^{-\lambda_1(s+Q)}}
-\frac{\int^{s+Q}_{-\infty}e^{\lambda_2y}g'(y)\,\mathrm{d}y }{\lambda_2e^{\lambda_2(s+Q)}};\\
c_2(s,s+Q)=&-\frac{\int_{s+Q}^{\infty}e^{-\lambda_1y}g'(y-Q)\,\mathrm{d}y}{\lambda_1 e^{-\lambda_1(s+Q)}}
+\frac{\int^{s+Q}_{-\infty}e^{\lambda_2y}g'(y-Q)\,\mathrm{d}y }{\lambda_2e^{\lambda_2(s+Q)}}.
\end{align*}
Then,
\begin{align}
\label{eq:other-exp}
a_{\lambda_1}(s,s+Q)-a_{\lambda_2}(s,s+Q)=c_1(s,s+Q)+c_2(s,s+Q).
\end{align}

By the L'Hospital Rule and denoting $-\frac{1}{\lambda_1^2}+\frac{1}{\lambda_2^2}$ by $\Delta$, we have
\begin{align}
\lim_{s\rightarrow -\infty}c_1(s,s+Q)=&\lim_{s\rightarrow -\infty}\left(-\Delta\right) \cdot g'(s+Q); \label{eq:diff-1}\\
\lim_{s\rightarrow -\infty}c_2(s,s+Q)=&\lim_{s\rightarrow -\infty}\Delta \cdot g'(s).\label{eq:diff-2}
\end{align}
When $\lim_{s\rightarrow -\infty}g'(s)>-\infty$, $\lim_{s\rightarrow -\infty}g'(s)=\lim_{s\rightarrow -\infty}g'(s+Q)$.
Therefore, it follows from \eqref{eq:a-1}, \eqref{eq:other-exp}-\eqref{eq:diff-2} that
\begin{align*}
\lim_{s\rightarrow -\infty}\left(a_1(s,s+Q)+a_{\lambda_1}(s,s+Q)-a_{\lambda_2}(s,s+Q)\right)<0.
\end{align*}
This makes \eqref{eq:b1-0} hold.

{\bf Case 2:} $\lim_{s\rightarrow -\infty}g'(s)=-\infty$.
In view of Assumption \ref{ass:h} (A3), $0<\frac{g'(s+Q)}{g'(s)}\leqslant 1$ for $s<-Q$.  First we have
\begin{align*}
&\limsup_{s\rightarrow -\infty}\frac{1}{g'(s+Q)}\left[\frac{\sigma^2(\lambda_1+\lambda_2)}{2}\left(K_1+kQ+g'(s+Q)\frac{Q}{\beta}\right)+c_1(s,s+Q)+c_2(s,s+Q)
\right]\\
& \ \ \ =\limsup_{s\rightarrow -\infty}\left[\frac{\sigma^2(\lambda_1+\lambda_2)}{2g'(s+Q)}(K_1+kQ)+\frac{\sigma^2(\lambda_1+\lambda_2)Q}{2\beta}+\frac{c_1(s,s+Q)}{g'(s+Q)}
+\frac{c_2(s,s+Q)}{g'(s)}\frac{g'(s)}{g'(s+Q)}\right]\\
& \ \ \ = \frac{\sigma^2(\lambda_1+\lambda_2)Q}{2\beta}-\Delta+\Delta \times \limsup_{s\rightarrow -\infty}\frac{g'(s)}{g'(s+Q)} \\
& \ \ \ \geqslant \frac{\sigma^2(\lambda_1+\lambda_2)Q}{2\beta}.
\end{align*}
Here the second equality is guaranteed by $\lim_{s\rightarrow -\infty} g'(s+Q)=-\infty$ and \eqref{eq:diff-1}-\eqref{eq:diff-2}.  With considering \eqref{eq:a-1}, the above equation also gives \eqref{eq:b1-0}.
Hence we complete the proof for the negativeness of $b_1(s,S)$.

Next consider $b_2(s,S)$. Since $S-s\in [Q, \infty)$, we just take $S=s+3Q$ and show that there exists $s$ with $s<-3Q$ such that
$b_2(s,s+3Q)<0$. For the first term in \eqref{eq:b-i} with $i=2$, similar to \eqref{eq:a-1},
\begin{align}
\label{eq:a-2}
\lim_{s\rightarrow-\infty}a_2(s,s+3Q)\leqslant &\lim_{s\rightarrow-\infty}
\frac{\sigma^2(\lambda_1+\lambda_2)}{2}\left(K_2+3kQ+g'(s+3Q)\frac{3Q}{\beta}\right)\\
<&\frac{\sigma^2(\lambda_1+\lambda_2)}{2}(K_2-3K_1)<0,\nonumber
\end{align}
where again the convexity of $g$ and Assumption \ref{ass:h} (A4) are applied in the first and second inequalities, respectively, and the last inequality follows from  $K_2\leqslant 2K_1$ in Assumption \ref{ass:K}.

Consider the second and third terms in \eqref{eq:b-i}, similar to \eqref{eq:a-lambda-1}-\eqref{eq:a-lambda-2},
\begin{align*}
a_{\lambda_1}(s,s+3Q)=&\frac{\int^\infty_{s+3Q}e^{-\lambda_1y} g'(y)\,\mathrm{d}y-
\int^\infty_{s+3Q}e^{-\lambda_1 y}g'(y-3Q)\,\mathrm{d} y}{\lambda_1 e^{-\lambda_1(s+3Q)}};\\
a_{\lambda_2}(s,s+3Q)=&\frac{\int_{-\infty}^{s+3Q}e^{\lambda_2 y}g'(y)\,\mathrm{d}y-\int_{-\infty}^{s+3Q}e^{\lambda_2y}g'(y-3Q)\,\mathrm{d}y}
{\lambda_2 e^{\lambda_2(s+3Q)}}.
\end{align*}
Then going along the same lines of the proof for the case $b_1(s,s+Q)$, we can prove that there exists $s$ with $s<-3Q$ such that $b_2(s,s+3Q)<0$.

Therefore, we have proved that there exists a finite $(s,S)$ such that $A_i(s,S)>\underline{A}$ and thus $A_i^*>\underline{A}$ for $i=1,2$.

We next prove $A_i(s,S)<\bar{A}$ for any $s<S$. We provide the details of the case when $s<S\leqslant0$, the proof of other cases is similar and thus is omitted. When $s<S\leqslant 0$, we can rewrite $A_i(s,S)$ as
\begin{align*}
A_i(s,S)=\frac{\lambda_2^2}{e^{-\lambda_2 S}-e^{-\lambda_2 s}}\Big[\frac{\sigma^2(\lambda_1+\lambda_2)}{2}(K_i+k(S-s))+r(s,S)\Big]+\bar{A},
\end{align*}
where
\begin{align*}
r(s,S)=&\Big(\frac{1}{\lambda_1}+\frac{1}{\lambda_2}\Big)(g(S)-g(s))
+\frac{1}{\lambda_1}\Big(e^{\lambda_1 S}\int_{S}^{\infty}e^{-\lambda_1 y}g'(y)\,\mathrm{d}y-e^{\lambda_1s}\int_{s}^{\infty}e^{-\lambda_1 y}g'(y)\,\mathrm{d}y\Big)\\
&+\frac{1}{\lambda_2}\Big(e^{-\lambda_2 S}\int_{S}^{0}e^{\lambda_2 y}g'(y)\,\mathrm{d}y-e^{-\lambda_2s}\int_{s}^{0}
e^{\lambda_2 y}g'(y)\,\mathrm{d}y\Big)\\
&-\frac{1}{\lambda_2^2}\Big(\lambda_1(e^{-\lambda_2S}-e^{-\lambda_2s})
\int_0^{\infty}e^{-\lambda_1y}g'(y)\,\mathrm{d}y\Big).
\end{align*}
If we can prove $r(s,S)\geqslant0$ for any $s<S\leqslant0$, since $e^{-\lambda_2 S}-e^{-\lambda_2 s}<0$ and
$\frac{1}{2}\sigma^2(\lambda_1+\lambda_2)(K_i+k(S-s))>0$, we must have $A_i(s,S)<\bar{A}$.
It remains to show $r(s,S)\geqslant0$ for any $s<S\leqslant0$.
We have
\begin{align}
\label{eq:dr/dS}
\frac{\partial r(s,S)}{\partial S}=e^{\lambda_1S}\int_{S}^{\infty}e^{-\lambda_1 y}g'(y)\,\mathrm{d}y -e^{-\lambda_2S}\int_{S}^{0}e^{\lambda_2 y}g'(y)\,\mathrm{d}y
+\frac{\lambda_1}{\lambda_2}e^{-\lambda_2S}
\int_0^{\infty}e^{-\lambda_1y}g'(y)\,\mathrm{d}y
\end{align}
and
\begin{align}
\label{eq:dr/dS-2}
\frac{\partial^2 r(s,S)}{\partial S^2}
=&\lambda_1e^{\lambda_1S}\int_{S}^{\infty}e^{-\lambda_1 y}g'(y)\,\mathrm{d}y+\lambda_2e^{-\lambda_2S}\int_{S}^{0}e^{\lambda_2 y}g'(y)\,\mathrm{d}y -\lambda_1e^{-\lambda_2S}
\int_0^{\infty}e^{-\lambda_1y}g'(y)\,\mathrm{d}y\\
=&\lambda_1\left[e^{\lambda_1S}-e^{-\lambda_2S}\right]\int_{0}^{\infty}e^{-\lambda_1 y}g'(y)\,\mathrm{d}y\nonumber\\
&+\lambda_1e^{\lambda_1S}\int_{S}^{0}e^{-\lambda_1 y}g'(y)\,\mathrm{d}y+\lambda_2e^{-\lambda_2S}\int_{S}^{0}e^{\lambda_2 y}g'(y)\,\mathrm{d}y\nonumber \\
\leqslant & 0,\nonumber
\end{align}
where the inequality follows from $S\in(s,0]$ and $g'(y)<0$ for $y<0$.
Furthermore, it follows from \eqref{eq:dr/dS} that
\[
\frac{\partial r(s,S)}{\partial S} \Big|_{S=0}=\frac{\lambda_1+\lambda_2}{\lambda_2}\int_0^{\infty}e^{-\lambda_1y}g'(y)\,\mathrm{d}y>0,
\]
which, together with \eqref{eq:dr/dS-2}, implies that
\[
\frac{\partial r(s,S)}{\partial S}>0\quad \text{for any $S\in(s,0]$}.
\]
Note that
\[
\lim_{S\to s}r(s,S)=0,
\]
thus we obtain $r(s,S)>0$ for any $s<S\leqslant0$.

($b$) As usual, the quasi-convexity of $v'_{A_i}(x)$ is established by considering its second order derivative. To this end,  using integration by parts, the definition of $v_{A_i}$ in \eqref{eq:v-Ai} implies that
\begin{align*}
v'_{A_i}(x)&=
\begin{cases}
\frac{1}{\beta} g^{\prime}(x)+\frac{2}{\sigma^{2}} \frac{1}{\lambda_{1}+\lambda_{2}} \cdot\big[\frac{1}{\lambda_{1}} \int_{x}^{\infty} e^{-\lambda_{1} y} g^{\prime \prime}(y) \,\mathrm{d}y \cdot e^{\lambda_{1} x}\\
\quad +\frac{1}{\lambda_{2}}\left(A_i-g^{\prime}(0+)-\int_{0}^{x} e^{\lambda_{2} y} g^{\prime \prime}(y) \,\mathrm{d}y\right) e^{-\lambda_{2} x}\big], & \text {for  $x\geqslant0$},\\[0.1in]
\frac{1}{\beta} g^{\prime}(x)+\frac{2}{\sigma^{2}} \frac{1}{\lambda_{1}+\lambda_{2}} \cdot\big[\frac{1}{\lambda_{1}}\left(g^{\prime}(0+)-g^{\prime}(0-)
+\int_{x}^{\infty} e^{-\lambda_{1} y} g^{\prime \prime}(y) \,\mathrm{d}y\right) e^{\lambda_{1} x} \\
\quad +\frac{1}{\lambda_{2}}\big(A_i-g^{\prime}(0-)+\int_{x}^{0} e^{\lambda_{2} y} g^{\prime \prime}(y) \,\mathrm{d}y\big) e^{-\lambda_{2} x}\big], & \text {for $x<0$},
\end{cases}
\end{align*}
\begin{align}
\label{eq:v''}
v''_{A_i}(x)&=
\begin{cases}
\frac{2}{\sigma^{2}} \frac{1}{\lambda_{1}+\lambda_{2}}\big[\int_{x}^{\infty} e^{-\lambda_{1} y} g^{\prime \prime}(y) \,\mathrm{d}y \cdot e^{\lambda_{1} x} \\
\quad-\big(A_i-g^{\prime}(0+)-\int_{0}^{x} e^{\lambda_{2} y} g^{\prime \prime}(y) \,\mathrm{d}y\big) e^{-\lambda_{2} x}\big], &\quad \text {for $x \geqslant 0$},\\[0.1in]
\frac{2}{\sigma^{2}} \frac{1}{\lambda_{1}+\lambda_{2}}\big[\big(g^{\prime}(0+)-g^{\prime}(0-)
+\int_{x}^{\infty} e^{-\lambda_{1} y} g^{\prime \prime}(y) \,\mathrm{d}y\big) e^{\lambda_{1} x} \\
\quad-\big(A_i-g^{\prime}(0-)+\int_{x}^{0} e^{\lambda_{2} y} g^{\prime \prime}(y) \,\mathrm{d}y\big) e^{-\lambda_{2} x}\big], &\quad \text {for  $x<0$},
\end{cases}
\end{align}
and
\begin{align}
\label{eq:v'''}
v'''_{A_i}(x)&=
\begin{cases}
\frac{2}{\sigma^{2}} \frac{1}{\lambda_{1}+\lambda_{2}}\big[\lambda_{1} \int_{x}^{\infty} e^{-\lambda_{1} y} g^{\prime \prime}(y) \,\mathrm{d}y \cdot e^{\lambda_{1} x} \\
\quad+\lambda_{2}\big(A_i-g^{\prime}(0+)-\int_{0}^{x} e^{\lambda_{2 y} y} g^{\prime \prime}(y) \,\mathrm{d}y\big) e^{-\lambda_{2} x}\big], &\quad \text {for $x \geqslant 0$},\\[0.1in]
\frac{2}{\sigma^{2}} \frac{1}{\lambda_{1}+\lambda_{2}}\big[\lambda_{1}
\big(g^{\prime}(0+)-g^{\prime}(0-)+\int_{x}^{\infty} e^{-\lambda_{1} y} g^{\prime \prime}(y) \,\mathrm{d}y\big) e^{\lambda_{1} x} \\
\quad+\lambda_{2}\big(A_i-g^{\prime}(0-)+\int_{x}^{0} e^{\lambda_{2} y} g^{\prime \prime}(y) \,\mathrm{d}y\big) e^{-\lambda_{2} x}\big], &\quad \text {for $x<0$}.
\end{cases}
\end{align}
Clearly, both $v_{A_i}'$ and $v_{A_i}''$ are continuous in $\mathbb{R}$.
When $A_i\in(\underline{A},\bar{A})$, we have that for any $x\geqslant 0$,
\begin{align}
\label{eq:v''>0}
v_{A_i}''(x)&> v''_{\bar{A}}(x)\\
&= \frac{2}{\sigma^{2}} \frac{1}{\lambda_{1}+\lambda_{2}}\left[\int_{x}^{\infty} e^{-\lambda_{1} y} g^{\prime \prime}(y) \,\mathrm{d}y \cdot e^{\lambda_{1} x}-\left(\int_{0}^{\infty}e^{-\lambda_1 y}g^{\prime\prime}(y)\,\mathrm{d}y-\int_{0}^{x} e^{\lambda_{2} y} g^{\prime \prime}(y) \,\mathrm{d}y\right) e^{-\lambda_{2} x}\right] \nonumber\\
&=\frac{2}{\sigma^{2}} \frac{1}{\lambda_{1}+\lambda_{2}}\left[\int_{x}^{\infty} e^{-\lambda_{1} y} g^{\prime \prime}(y) \,\mathrm{d}y \cdot\big( e^{\lambda_{1} x}-e^{-\lambda_2x}\big)+e^{-\lambda_2 x}\int_{0}^{x} \big(e^{\lambda_2 y}-e^{-\lambda_{1} y}\big) g^{\prime \prime}(y) \,\mathrm{d}y \right]\nonumber\\
&\geqslant0,\nonumber
\end{align}
where the first inequality follows from \eqref{eq:v''}, the second equality follows from the definition of $\bar{A}$ in \eqref{eq:bar-A}.
Furthermore, it follows from \eqref{eq:v''} that
\begin{align*}
\lim _{x \rightarrow-\infty} \frac{v''_{A_i}(x)}{e^{-\lambda_{2} x}}
&=\frac{2}{\sigma^{2}} \frac{1}{\lambda_{1}+\lambda_{2}} \lim _{x \rightarrow-\infty} \frac{g'(0+)-g'(0-)+\int_{x}^{\infty} e^{-\lambda_{1} y} g''(y) \,\mathrm{d}y}{e^{-(\lambda_{1}+\lambda_{2}) x}}-\frac{2}{\sigma^{2}} \frac{1}{\lambda_{1}+\lambda_{2}}(A_i-\underline{A})\\
&=\frac{2}{\sigma^{2}} \frac{1}{\lambda_{1}+\lambda_{2}} \lim _{x \rightarrow-\infty} \frac{\lambda_1\int_{x}^{\infty} e^{-\lambda_{1} y} g'(y) \,\mathrm{d}y-e^{-\lambda_1x}g'(x)}{e^{-(\lambda_{1}+\lambda_{2}) x}}-\frac{2}{\sigma^{2}} \frac{1}{\lambda_{1}+\lambda_{2}}(A_i-\underline{A}) \\
&=\frac{2}{\sigma^{2}} \frac{1}{\lambda_{1}+\lambda_{2}}(\underline{A}-A_i)\\
&< 0,
\end{align*}
where the second equality follows from integration by parts and the polynomially bounded property of $g'(x)$ in Remark \ref{rem:h} ($b$), and the third equality follows from the polynomially bounded property of $g'(x)$.
Thus,
\begin{equation}
\label{eq:lim-v''}
\lim _{x \rightarrow-\infty}v''_{A_i}(x)=-\infty.
\end{equation}
In addition, it follows from \eqref{eq:v'''} that for $x<0$,
\begin{align*}
v'''_{A_i}(x)
=&\frac{2}{\sigma^{2}} \frac{1}{\lambda_{1}+\lambda_{2}}\Bigg[\lambda_{1}
\left(g^{\prime}(0+)-g^{\prime}(0-)+\int_{x}^{\infty} e^{-\lambda_{1} y} g^{\prime \prime}(y) \,\mathrm{d}y\right) e^{\lambda_{1} x}\\
&+\lambda_{2}\left(A_i-g^{\prime}(0-)+\int_{x}^{0} e^{\lambda_{2} y} g^{\prime \prime}(y) \,\mathrm{d}y\right) e^{-\lambda_{2} x}\Bigg].
\end{align*}
If $A_i\in[g'(0-),\bar{A})$, we have
\[
v'''_{A_i}(x)>0\quad \text{for $x<0$},
\]
which, together with \eqref{eq:v''>0}-\eqref{eq:lim-v''}, implies that there exists a unique $x_{A_i}^*$ with $x_{A_i}^*<0$ satisfying
\begin{align*}
v_{A_i}''(x)
\begin{cases}
<0, & x<x_{A_i}^*,\\
=0, & x=x_{A_i}^*,\\
>0, & x>x_{A_i}^*,
\end{cases}
\end{align*}
i.e., the function $v_{A_i}'$ is strictly quasi-convex with a unique negative minimizer $x_{A_i}^*$.
If $A_i\in(\underline{A},g'(0-))$, since $A_i-g^{\prime}(0-)+\int_{x}^{0} e^{\lambda_{2} y} g^{\prime \prime}(y) \,\mathrm{d}y$ is decreasing  in $x\in(-\infty,0)$,
\begin{align*}
\lim_{x\to-\infty}\left(A_i-g^{\prime}(0-)+\int_{x}^{0} e^{\lambda_{2} y} g^{\prime \prime}(y) \,\mathrm{d}y\right)
&=A_i-\underline{A}>0
\quad \text{and}\\
\lim_{x\to0-}\left(A_i-g^{\prime}(0-)+\int_{x}^{0} e^{\lambda_{2} y} g^{\prime \prime}(y) \,\mathrm{d}y\right)
&=A_i-g^{\prime}(0-)<0,
\end{align*}
there exists an $x_0$ with $x_0<0$ such that
\begin{align*}
A_i-g^{\prime}(0-)+\int_{x}^{0} e^{\lambda_{2} y} g^{\prime \prime}(y) \,\mathrm{d}y&>0\quad \text{for $x<x_0$}\quad\text{and}\\
A_i-g^{\prime}(0-)+\int_{x}^{0} e^{\lambda_{2} y} g^{\prime \prime}(y) \,\mathrm{d}y&\leqslant0\quad \text{for $x_0\leqslant x<0$},
\end{align*}
which imply that
\begin{align}
v_{A_i}''(x)&\geqslant\frac{2}{\sigma^{2}} \frac{1}{\lambda_{1}+\lambda_{2}}\big(g^{\prime}(0+)-g^{\prime}(0-)
+\int_{x}^{\infty} e^{-\lambda_{1} y} g^{\prime \prime}(y) \,\mathrm{d}y\big) e^{\lambda_{1} x} >0\quad \text{for $x_0\leqslant x<0$
 and }\label{eq:v''>0-2}\\
v'''_{A_i}(x)&>0\quad \text{for $x<x_0$},\label{eq:v'''>0-2}
\end{align}
where the second inequality in \eqref{eq:v''>0-2} follows from that for any $x\leqslant 0$
\begin{align}
\label{eq:>0-convex}
g^{\prime}(0+)-g^{\prime}(0-)
+\int_{x}^{\infty} e^{-\lambda_{1} y} g^{\prime \prime}(y) \,\mathrm{d}y
\geqslant g^{\prime}(0+)-g^{\prime}(0-)
+\int_{0}^{\infty} e^{-\lambda_{1} y} g^{\prime \prime}(y) \,\mathrm{d}y>0
\end{align}
by noting that one of $g^{\prime}(0+)$ and $\int_{0}^{\infty} e^{-\lambda_{1} y} g^{\prime \prime}(y) \,\mathrm{d}y$ must be strictly larger than 0 (see Assumption \ref{ass:h}).
Therefore, it follows from \eqref{eq:v''>0}-\eqref{eq:v'''>0-2} that
there exists a unique $x_{A_i}^*$ with $x_{A_i}^*<x_0<0$ satisfying
\begin{align*}
v_{A_i}''(x)
\begin{cases}
<0, & x<x_{A_i}^*,\\
=0, & x=x_{A_i}^*,\\
>0, & x>x_{A_i}^*,
\end{cases}
\end{align*}
i.e., the function $v_{A_i}'$ is strictly quasi-convex with a unique negative minimizer $x_{A_i}^*$.

($c$)  We only prove the case when $i=1$, and the proof of the case when $i=2$ is similar and thus is omitted.
Note that $A_1^*\in(\underline{A},\bar{A})$ and $S-s\in(0,Q]$, thus if we can prove that
\begin{align}
\label{eq:lim-A1}
\lim_{s\to-\infty}A_1(s,S)=\underline{A},\quad \lim_{S\downarrow s}A_1(s,S)=-\infty,\quad \text{and}\quad \lim_{S\to\infty}A_1(s,S)= -\infty,
\end{align}
there must exist at least one finite local maximizer of $\mathcal{OP}_1$, i.e., there exists at least one finite solution, denoted by $(s_1,S_1)$, of the KKT conditions satisfying $A_1(s_1,S_1)\in(\underline{A},\bar{A})$.
We next prove \eqref{eq:lim-A1}. First, recall the expression of $A_i(s,S)$ in \eqref{eq:Ai(s,S)-re}, it is similar to the analysis in the proof of part ($a$), the polynomial boundness of $g'$ implies that
\[
\lim_{s\to-\infty}A_1(s,S)=\lim_{s\to-\infty}\frac{\lambda_2^2}{e^{-\lambda_2 S}-e^{-\lambda_2 s}}b_1(s,S)+\underline{A}=\underline{A},
\]
i.e., we obtain the first part in \eqref{eq:lim-A1}.
To prove the last two ones in \eqref{eq:lim-A1}, we rewrite $A_1(s,S)$ as
\begin{align*}
A_1(s,S)&=\frac{\lambda_2^2}{e^{-\lambda_2 S}-e^{-\lambda_2 s}}\Big[d_1(s,S)+\frac{\sigma^2(\lambda_1+\lambda_2)}{2}(K_1+k(S-s))\Big],
\end{align*}
with
\begin{align}
\label{eq:c1}
d_1(s,S)=&(g(S)-g(s))\Big(\frac{1}{\lambda_1}+\frac{1}{\lambda_2}\Big)
+\frac{1}{\lambda_1}\left(e^{\lambda_1S}\int_S^{\infty}e^{-\lambda_1y}g'(y)\,\mathrm{d}y
-e^{\lambda_1s}\int_s^{\infty}e^{-\lambda_1y}g'(y)\,\mathrm{d}y\right)\\
&-\frac{1}{\lambda_2}\left(e^{-\lambda_2S}\int^S_{0}e^{\lambda_2y}g'(y)\,\mathrm{d}y
-e^{-\lambda_2s}\int^s_{0}e^{\lambda_2y}g'(y)\,\mathrm{d}y\right).\nonumber
\end{align}
Since
\begin{equation}
\label{eq:lim-S}
\lim_{S\downarrow s}d_1(s,S)=0,
\end{equation}
we have $\lim_{S\downarrow s}A_1(s,S)=-\infty$, i.e., the second one in \eqref{eq:lim-A1} holds.
Furthermore, it follows from \eqref{eq:c1} that for $S>0$,
\[
\frac{\partial d_1(s,S)}{\partial S}=e^{\lambda_1 S}\int_S^{\infty} e^{-\lambda_1 y}g'(y)\,\mathrm{d}y
+e^{-\lambda_2 S}\int_0^{S} e^{\lambda_2 y}g'(y)\,\mathrm{d}y>0,
\]
which, together with \eqref{eq:lim-S}, implies that for any $0<s<S$,
\begin{equation}
\label{eq:c1>0}
d_1(s,S)>0.
\end{equation}
Note that when $S-s\in(0,Q]$,
\[
0\geqslant\lim_{S\to\infty}\left[e^{-\lambda_2 S}-e^{-\lambda_2 s}\right]=\lim_{S\to\infty}e^{-\lambda_2 S}\left[1-e^{\lambda_2 (S-s)}\right]\geqslant \lim_{S\to\infty}e^{-\lambda_2 S}\left[1-e^{\lambda_2 Q}\right]=0,
\]
and then $\lim_{S\to\infty}[e^{-\lambda_2 S}-e^{-\lambda_2 s}]=0$, which yields
\[
\lim_{S\to\infty}A_1(s,S)\leqslant \lim_{S\to\infty}\frac{\lambda_2^2}{e^{-\lambda_2 S}-e^{-\lambda_2 s}}\frac{\sigma^2(\lambda_1+\lambda_2)\big(K_1+k(S-s)\big)}{2}=-\infty,
\]
where the inequality follows from \eqref{eq:c1>0}. Therefore, we have proven the last one in  \eqref{eq:lim-A1}.

It remains to show the uniqueness of $(s_1,S_1)$.
It follows from the definition of $A_1$ in \eqref{eq:A(s,S)} and $v_{A_1}$ in \eqref{eq:v-Ai} that
\begin{align}
\frac{\partial A_1(s_1,S_1)}{\partial s}&=-\frac{\lambda_2^2}{e^{-\lambda_2 S_1}-e^{-\lambda_2s_1}}\frac{\sigma^2(\lambda_1+\lambda_2)}{2}\Big(v_{A_1(s_1,S_1)}'(s_1)+k\Big),\label{eq:dA1/ds}\\
\frac{\partial A_1(s_1,S_1)}{\partial S}&=\frac{\lambda_2^2}{e^{-\lambda_2 S_1}-e^{-\lambda_2s_1}}\frac{\sigma^2(\lambda_1+\lambda_2)}{2}\Big(v_{A_1(s_1,S_1)}'(S_1)+k\Big).\label{eq:dA1/dS}
\end{align}
Furthermore, the KKT conditions  \eqref{eq:KKTcondition-1}-\eqref{eq:KKTcondition-2} imply
\[
\frac{\partial A_1(s_1,S_1)}{\partial s}=-\frac{\partial A_1(s_1,S_1)}{\partial S},
\]
which, together with \eqref{eq:dA1/ds}-\eqref{eq:dA1/dS}, implies
\begin{equation}
\label{eq:v1=v1}
v_{A_1(s_1,S_1)}'(s_1)=v_{A_1(s_1,S_1)}'(S_1).
\end{equation}
Since $A_1(s_1,S_1)\in(\underline{A},\bar{A})$, it follows from part ($b$) that $v_{A_1(s_1,S_1)}'$ is strictly quasi-convex, which yields that there exists at most one pair $(s_1,S_1)$ satisfying \eqref{eq:v1=v1}, i.e.,
such $(s_1,S_1)$ is unique (because we have obtained the existence).

The fact $A_1^*\in(\underline{A},\bar{A})$ and the uniqueness of $(s_1,S_1)$ for $\mathcal{OP}_1$ such that local maximum $A_1(s_1,S_1)\in (\underline{A},\bar{A})$ imply that
$A_1^*=A_1(s_1,S_1)$ and $v_1(x)=v_{A_1^*}(x)=v_{A_1(s_1,S_1)}(x)$.

($d$) If $S_1-s_1<Q$, the complementary slackness condition \eqref{eq:KKTcondition-3} implies $\eta_1=0$, which together with
\eqref{eq:dA1/ds}-\eqref{eq:dA1/dS}, implies that \eqref{eq:KKTcondition-1}-\eqref{eq:KKTcondition-2} are equivalent to
\[
v_{A_1(s_1,S_1)}'(s_1)=v_{A_1(s_1,S_1)}'(S_1)=-k,
\]
which implies $v_1'(s_1) =v_1'(S_1)=-k$ by noting $A_1^*=A_1(s_1,S_1)$.

If $S_1-s_1=Q$, the complementary slackness condition \eqref{eq:KKTcondition-3} implies $\eta_1\geqslant0$, which together with
\eqref{eq:dA1/ds}-\eqref{eq:dA1/dS}, implies that \eqref{eq:KKTcondition-1}-\eqref{eq:KKTcondition-2} are equivalent to
\[
v_{A_1(s_1,S_1)}'(s_1)=v_{A_1(s_1,S_1)}'(S_1)\leqslant -k.
\]
i.e., $v_1'(s_1) =v_1'(S_1)\leqslant-k$.

($e$) The proof for this part is very similar to the one of ($d$) and thus is omitted.

$(f)$ It follows from part $(b)$ that $v_i''(x_i^*)=0$, which, together with \eqref{eq:v''}, implies that $x_i^*<0$ is the solution of equation
\begin{equation}\label{eq:xA*}
\Big(g^{\prime}(0+)-g^{\prime}(0-)
+\int_{x}^{\infty} e^{-\lambda_{1} y} g^{\prime \prime}(y) \,\mathrm{d}y\Big) e^{\lambda_{1} x}
=\Big(A_i^*-g^{\prime}(0-)+\int_{x}^{0} e^{\lambda_{2} y} g^{\prime \prime}(y)\,\mathrm{d}y\Big) e^{-\lambda_{2} x}.
\end{equation}
Furthermore, we have
\begin{align}
-k\geqslant& v'_1(s_1)\label{eq:g1xA*}\\
>&v'_1(x_1^*)\nonumber\\
=& \frac{1}{\beta} g^{\prime}(x_1^*)+\frac{2}{\sigma^{2}} \frac{1}{\lambda_{1}+\lambda_{2}} \Big[\frac{1}{\lambda_{1}}\Big(g^{\prime}(0+)-g^{\prime}(0-)
+\int_{x_1^*}^{\infty} e^{-\lambda_{1} y} g^{\prime \prime}(y) \,\mathrm{d}y\Big) e^{\lambda_{1} x_1^*}\nonumber\\
 &+\frac{1}{\lambda_{2}}\Big(A_1^*-g^{\prime}(0-)+\int_{x_1^*}^{0} e^{\lambda_{2} y} g^{\prime \prime}(y) \,\mathrm{d}y\Big) e^{-\lambda_{2} x_1^*}\Big]\nonumber\\
=& \frac{1}{\beta} g^{\prime}(x_1^*)+\frac{e^{\lambda_{1} x_1^*}}{\beta}
\Big[g^{\prime}(0+)-g^{\prime}(0-)
+\int_{x_1^*}^{\infty} e^{-\lambda_{1} y} g^{\prime \prime}(y) \,\mathrm{d}y\Big]\nonumber\\
>&\frac{1}{\beta} g^{\prime}(x_1^*)\nonumber,
\end{align}
where the first inequality follows from part ($d$), the second inequality follows from $s_1<x_1^*$ and that $v_1'$ is strictly quasi-convex with unique minimizer $x_1^*$,
the second  equality follows from \eqref{eq:xA*}, and the last inequality follows from \eqref{eq:>0-convex}.
Then,  \eqref{eq:g1xA*} and the convexity of $g(x)$ imply that
\begin{equation*}
g'(x)+\beta k\leqslant g'(x_1^*)+\beta k<0 \quad \text{for $x\leqslant x_1^*$.}
\end{equation*}
When $S_2-s_2>Q$, it follows from part $(e)$ that $-k=v_2'(s_2)$, then it is similar to \eqref{eq:g1xA*} that
\[
-k>\frac{1}{\beta} g^{\prime}(x_2^*),
\]
and then we have
\begin{equation*}
g'(x)+\beta k\leqslant g'(x_2^*)+\beta k<0 \quad \text{for $x\leqslant x_2^*$}.
\end{equation*}

\subsection{Proof of Lemma \ref{lem:bar-S}}
$(a)$ It can be concluded from Lemma \ref{lem:sS-A} $(d)$  and Remark \ref{rem:vi} $(a)$ that $v_1'(S_1)\leqslant -k$ and $v_1''(x)>0$ for $x>S_1>x_1^*$.  Thus, it follows from $\lim_{x\rightarrow +\infty}v_{1}'(x)>0$ (see \eqref{eq:lim-v1'}) that there exists a unique solution $\bar{S}\in[S_1,\infty)$ to equation $v_{1}'(x)=-k$.

$(b)$  Lemma \ref{lem:sS-A} ($d$) and Remark \ref{rem:vi} ($a$)  imply that
\begin{equation}
\label{eq:H'(x)<0}
\mathcal{H}'(x)=v_1'(x+Q)+k<0\quad \text{ for $x\in(s_1-Q,S_1-Q)$.}
\end{equation}
Thus, if we can prove that
\begin{equation}
\label{eq:delta}
\mathcal{H}(s_1-Q)\geqslant 0 \quad \text{and} \quad \mathcal{H}(S_1-Q)\leqslant 0,
\end{equation}
then there exists a unique point $\underline{s}$ in $[s_1-Q,S_1-Q]$, such that $\mathcal{H}(\underline{s})=0$.
We next prove \eqref{eq:delta} by two cases: $v_1'(s_1)=v_1'(S_1)=-k$ and $v_1'(s_1)=v_1'(S_1)<-k$.

\noindent\textsf{Case 1}: When $v_1'(s_1)=v_1'(S_1)=-k$, it follows from the definition of $\bar{S}$ in part (a) and $v_1''(x)>0$ for $x>S_1>x_1^*$ (Remark \ref{rem:vi} $(a)$) that $\bar{S}=S_1$, and then
\begin{align*}
\mathcal{H}(s_1-Q)
=v_1(s_1)-(v_1(S_1)+K_1+k(S_1-s_1))+2K_1-K_2
=2K_1-K_2
\geqslant0,
\end{align*}
where $v_1(s_1)=v_1(S_1)+K_1+k(S_1-s_1)$ in the last equality follows from $v_1(x)=v_1^{(s_1,S_1)}(x)$ for $x\geqslant s_1$  (see \eqref{eq:solution-vi}-\eqref{eq:v=v+K}, \eqref{eq:v-Ai}, and \eqref{eq:A-star}) and \eqref{eq:boundary-condition} with $s=s_1$ and $S=S_1$.
Furthermore, we have
\begin{align}
\label{eq:H(S1-Q)<0}
\mathcal{H}(S_1-Q)
=v_1(S_1)+K_1+kQ-(v_1(\bar{S})+K_2+k(\bar{S}-S_1+Q))
=K_1-K_2
<0.
\end{align}
Thus, \eqref{eq:delta} is proven. In addition, $\mathcal{H}(\underline{s})=0$, \eqref{eq:H'(x)<0}, and \eqref{eq:H(S1-Q)<0} imply \eqref{eq:underline-s<}.

\noindent\textsf{Case 2}: When $v_1'(s_1)=v_1'(S_1)<-k$,  it follows from the definition of $\bar{S}$ that $\bar{S}>S_1$ and
\begin{equation}
\label{eq:dv1<-k}
v_1'(x)<-k\quad \text{for any $x\in(S_1,\bar{S})$}.
\end{equation}
Thus,
\begin{align*}
\mathcal{H}(s_1-Q)
&=v_1(s_1)-v_1(\bar{S})+K_1-K_2-k(\bar{S}-s_1)\\
&=v_1(S_1)-v_1(\bar{S})-k(\bar{S}-S_1)+2K_1-K_2\\
&=(v_1'(\theta)+k)(S_1-\bar{S})+2K_1-K_2\\
&>0,
\end{align*}
where the second equality follows from $v_1(s_1)=v_1(S_1)+K_1+k(S_1-s_1)$, and the inequality follows from $2K_1\geqslant K_2$ and \eqref{eq:dv1<-k}.
Furthermore, it follows from Lemma \ref{lem:sS-A} ($d$) that $S_1-s_1= Q$ and then
\begin{align*}
\mathcal{H}(S_1-Q)=\mathcal{H}(s_1)=\big(v_1(S_1)+K_1+kQ\big)-\big(v_1(\bar{S})+K_2+k\cdot(\bar{S}-s_1)\big).
\end{align*}
Note that $v_1(S_1)+K_1+kQ$ is the discounted cost at $x=s_1$ under $(s_1,S_1)$ policy, i.e.,
\[
\mathsf{DC}(s_1,(s_1,S_1))=v_1(S_1)+K_1+k\cdot(S_1-s_1)=v_1(S_1)+K_1+kQ,
\]
while, $v_1(\bar{S})+K_2+k(\bar{S}-s_1)$ can be regarded as the discounted cost at level $s_1$ under an admissible policy $\phi$ that increases the inventory to level $\bar{S}$ (from level $s_1$) immediately at time 0 and then follows $(s_1,S_1)$ policy, i.e.,
\[
\mathsf{DC}(s_1,\phi)=v_1(\bar{S})+K_2+k(\bar{S}-s_1).
\]
It follows from Theorem \ref{thm2} that $\mathsf{DC}(s_1,(s_1,S_1))\leqslant \mathsf{DC}(s_1,\phi)$, and thus $\mathcal{H}(S_1-Q)=\mathsf{DC}(s_1,(s_1,S_1))-\mathsf{DC}(s_1,\phi)\leqslant 0$.
\qed

\subsection{Supplement of Section \ref{sec:proof-thm2}}
\label{app:supplement-thm2}

In this part, we prove that function $V_1$ defined in \eqref{eq:V1-def} satisfies all conditions in Corollary \ref{cor:lowerbound}.

First, the definition of $V_1$ and \eqref{eq:v1(bar-s)} imply that $V_1$ and $V_1'$ are continuous in whole $\mathbb{R}$, $V_1''$ is continuous except at $\bar{s}_1$,  $V_1''(\bar{s}_1-)=\lim_{x\rightarrow \bar{s}_1-}V_1''(x)=0$, and $V_1''(\bar{s}_1+)=\lim_{x\rightarrow \bar{s}_1+}V_1''(x)=v_1''(\bar{s}_1)$.

Check condition \eqref{eq:lowerbound-1}. The definitions of $V_1$ and $v_1$ imply that for $x\in[\bar{s}_1,\infty)$,
\begin{equation}
\label{eq:Gamma=0-V1}
\Gamma V_1(x)-\beta V_1(x)+g(x)=\Gamma v_1(x)-\beta v_1(x)+g(x)=0.
\end{equation}
For $x\in(-\infty,\bar{s}_1)$, we have
\[
\Gamma V_1(x)-\beta V_1(x)+g(x)=\mu k-\beta\big(v_1(\bar{s}_1)+k(\bar{s}_1-x)\big)+g(x).
\]
Let $\Phi_3(x)=\mu k-\beta\big(v_1(\bar{s}_1)+k(\bar{s}_1-x)\big)+g(x)$, then for $x<\bar{s}_1$,
\[
\Phi_3'(x)=g'(x)+\beta k<0,
\]
where the inequality follows from Lemma \ref{lem:sS-A} ($f$)
and the fact that $\bar{s}_1\leqslant s_1<x_1^*$.
Thus, for $x<\bar{s}_1$,
\begin{equation}
\label{eq:Gamma>Phi-2}
\Gamma V_1(x)-\beta V_1(x)+g(x)=\Phi_3(x)>\Phi_3(\bar{s}_1).
\end{equation}
Taking $x=\bar{s}_1$ in \eqref{eq:Gamma=0-V1}, we have
\[
0=\Gamma v_1(\bar{s}_1)-\beta v_1(\bar{s}_1)+g(\bar{s}_1)=\frac{1}{2}v_1''(\bar{s}_1)-\mu v_1'(\bar{s}_1)-\beta v_1(\bar{s}_1)+g(\bar{s}_1)=\frac{1}{2}v_1''(\bar{s}_1)+\Phi_3(\bar{s}_1),
\]
which, together with $v_1''(\bar{s}_1)<0$ (since $v_1'$ is strictly quasi-convex and $\bar{s}_1<x_1^*$), implies
$\Phi_3(\bar{s}_1)>0$.
Therefore, \eqref{eq:Gamma>Phi-2} implies that for $x<\bar{s}_1$,
$\Gamma V_1(x)-\beta V_1(x)+g(x)>0$.

Check condition \eqref{eq:lowerbound-2}. This is done by considering three cases: $x_1<x_2<\bar{s}_1$, $\bar{s}_1\leqslant x_1<x_2$, and $x_1<\bar{s}_1\leqslant x_2$.

\noindent$\mathsf{Case\ 1:}$ If $x_1<x_2<\bar{s}_1$, we have $V_1(x_2)-V_1(x_1)=-k\cdot(x_2-x_1)>-K(x_2-x_1)-k\cdot(x_2-x_1)$.

\noindent$\mathsf{Case\ 2:}$ If $\bar{s}_1\leqslant x_1<x_2$, we have
\begin{equation}
\label{eq:V1=}
V_1(x_2)-V_1(x_1)=v_1(x_2)-v_1(x_1)=\int_{x_1}^{x_2}[v_1'(y)+k]\,\mathrm{d}y-k\cdot (x_2-x_1).
\end{equation}

\noindent$\mathsf{Subcase\ 2.1:}$ When $S_1-s_1<Q$, we have $v_1'(s_1)=v_1'(S_1)=-k$ and then
\begin{align*}
\int_{x_1}^{x_2}[v_1'(y)+k]\,\mathrm{d}y&\geqslant \int_{s_1}^{S_1}[v_1'(y)+k]\,\mathrm{d}y\\
&=v_1(S_1)-v_1(s_1)+k\cdot(S_1-s_1)\\
&=-K_1\\
&\geqslant -K(x_2-x_1),
\end{align*}
which implies $V_1(x_2)-V_1(x_1)\geqslant -K(x_2-x_1)-k\cdot (x_2-x_1)$, i.e., condition \eqref{eq:lowerbound-2} holds.

\noindent$\mathsf{Subcase\ 2.2:}$ When $S_1-s_1=Q$, we have $v_1'(s_1)=v_1'(S_1)\leqslant -k$.
If $x_2-x_1\leqslant Q$, we have
\begin{align*}
\int_{x_1}^{x_2}[v_1'(y)+k]\,\mathrm{d}y
&\geqslant \int_{s_1}^{S_1}[v_1'(y)+k]\,\mathrm{d}y\\
&=v_1(S_1)-v_1(s_1)+k\cdot(S_1-s_1)\\
&=-K_1\\
&=-K(x_2-x_1),
\end{align*}
which, together with \eqref{eq:V1=}, implies $V_1(x_2)-V_1(x_1)\geqslant -K(x_2-x_1)-k\cdot (x_2-x_1)$.
Now consider the case when $x_2-x_1>Q$. We first prove that $S_1-s_1=Q$ implies
\begin{equation}
\label{eq:v2'=-k}
v_2'(s_2)=v_2'(S_2)=-k.
\end{equation}
It follows from Lemma \ref{lem:sS-A} ($d$) that $v_1'(s_1)=v_1'(S_1)\leqslant -k$,
which, together with $A_1^*>A_2^*$ and \eqref{eq:vi-derivate}, implies that
\[
v_2'(s_1)<-k\quad\text{and}\quad v_2'(S_1)<-k.
\]
Recall $v_2'(s_2)=v_2'(S_2)\geqslant-k$ (see Lemma \ref{lem:sS-A} ($e$)), the quasi-convexity of $v_2'$ implies $S_2-s_2>S_1-s_1=Q$, and then \eqref{eq:v2'=-k} holds.
Thus, we have
\begin{align}
\label{eq:int>-K}
\int_{x_1}^{x_2}[v_1'(y)+k]\,\mathrm{d}y
\geqslant  \int_{x_1}^{x_2}[v_2'(y)+k]\,\mathrm{d}y
\geqslant \int_{s_2}^{S_2}[v_2'(y)+k]\,\mathrm{d}y=-K_2=-K(x_2-x_1),
\end{align}
where the first inequality follows from $A_1^*>A_2^*$ and \eqref{eq:vi-derivate}, the second inequality follows from \eqref{eq:v2'=-k} and the quasi-convexity of $v_2'$, and the last equality follows from $x_2-x_1>Q$.
Therefore, \eqref{eq:V1=} and \eqref{eq:int>-K} imply condition \eqref{eq:lowerbound-2}.

\noindent$\mathsf{Case\ 3:}$ If $x_1<\bar{s}_1\leqslant x_2$, we have
\begin{align*}
V_1(x_2)-V_1(x_1)
&=v_1(x_2)-v_1(\bar{s}_1)+k\cdot(x_2-\bar{s}_1)-k\cdot(x_2-x_1)\\
&\geqslant-K(x_2-\bar{s}_1)-k\cdot(x_2-x_1)\\
&\geqslant-K(x_2-x_1)-k\cdot(x_2-x_1),
\end{align*}
where the first inequality follows from $\mathsf{Case\ 2}$, and the last inequality follows from that $K(\cdot)$ is increasing.

Finally, the proof of conditions \eqref{eq:lowerbound-3}-\eqref{eq:lowerbound-4} is very similar to that in the proof of Theorem \ref{thm1}, thus we omit the details.

\subsection{Supplement of Section \ref{sec:proof-thm4}}
\label{app:supplement-thm4}

In this part, we check that $\bar{V}_1$ defined in \eqref{eq:bar-V1} satisfies conditions \eqref{eq:lowerbound-1}-\eqref{eq:lowerbound-4}.

We first check condition \eqref{eq:lowerbound-1}. For $x\in[S_1-Q, \infty)$, since $\bar{V}_1(x)$ is the same as $V_1(x)$ defined in Theorem \ref{thm2}, condition \eqref{eq:lowerbound-1} has been proven in the proof of Theorem \ref{thm2}, and in particular,
\begin{align}
\label{eq:Gamma-barV1}
\Gamma \bar{V}_1(x)-\beta \bar{V}_1(x)+g(x)\geqslant0 \quad \text{for $x>S_1-Q$}.
\end{align}
We next check the cases when
$x\in [\underline{s},S_1-Q)$ and $x\in(-\infty,\underline{s})$.

\noindent\textsf{Case 1:} $x\in [\underline{s},S_1-Q)$. In this case, we have
\begin{align*}
&\Gamma \bar{V}_1(x)-\beta \bar{V}_1(x)+g(x)\\
&\quad=\big[\Gamma v_1(x+Q)-\beta v_1(x+Q)+g(x+Q)\big] +g(x)-g(x+Q)-\beta(K_1+kQ)\\
&\quad\geqslant  g(x)-g(x+Q)-\beta(K_1+kQ),
\end{align*}
where the inequality follows from \eqref{eq:Gamma-barV1} at $x+Q\geqslant \underline{s}+Q\geqslant s_1$.
Let $\Phi_4(x)=g(x)-g(x+Q)-\beta(K_1+kQ)$, we have $\Phi_4'(x)=g'(x)-g'(x+Q)\leqslant 0$ for $x\in [\underline{s},S_1-Q)$. If we can prove
\begin{equation}
\label{eq:Phi4}
\Phi_4(S_1-Q)>0,
\end{equation}
then we get $\Phi_4(x)\geqslant0$ for $x\in [\underline{s},S_1-Q)$, i.e., $\Gamma \bar{V}_1(x)-\beta \bar{V}_1(x)+g(x)\geqslant 0$ for $x\in [\underline{s},S_1-Q)$. We next prove \eqref{eq:Phi4}.  It follows from \eqref{eq:Gamma-barV1} that
\begin{align*}
0&\leqslant \Gamma \bar{V}_1((S_1-Q)+)-\beta \bar{V}_1(S_1-Q)+g(S_1-Q)\\
&=\mu k -\beta \big(v_1(S_1)+K_1+kQ\big)+g(S_1-Q)\\
&\leqslant -\mu v_1'(S_1)-\beta \big(v_1(S_1)+K_1+kQ\big)+g(S_1-Q),
\end{align*}
where the last inequality follows from $v_1'(S_1)\leqslant -k$. Thus, we have
\[
g(S_1-Q)-\beta \big(K_1+kQ\big)\geqslant \mu v_1'(S_1)+\beta v_1(S_1),
\]
which implies that
\begin{align*}
\Phi_4(S_1-Q)
&=g(S_1-Q)-g(S_1)-\beta(K_1+kQ)\\
&\geqslant \mu v_1'(S_1)+\beta v_1(S_1)-g(S_1)\\
&=\frac{1}{2}\sigma^2 v_1''(S_1)\\
&>0,
\end{align*}
where the last equality follows from $\Gamma v_1(S_1)-\beta v_1(S_1)+g(S_1)=0$ and the last inequality follows from the quasi-convexity of $v_1'$ and that $S_1$ is on the right of the minimum point.

\noindent\textsf{Case 2:} $x\in(-\infty,\underline{s})$. In this case, we have
\begin{align*}
\Gamma \bar{V}_1(x)-\beta \bar{V}_1(x)+g(x)=\mu k-\beta \big(v_1(\bar{S})+K_2+k\cdot(\bar{S}-x)\big)+g(x).
\end{align*}
Let $\Phi_5(x)=\mu k-\beta \big(v_1(\bar{S})+K_2+k\cdot(\bar{S}-x)\big)+g(x)$, we have
\[
\Phi_5(\underline{s}) =\Xi(\underline{s})\geqslant 0\quad \text{and}\quad
\Phi_5'(x)=\beta k+g'(x)<0 \quad \text{for $x<\underline{s}<0$},
\]
where the inequality in the first part follows from the assumption in Theorem \ref{thm4}, and the inequality in the second part follows from Lemma \ref{lem:sS-A} ($f$) and $x<\underline{s}\leqslant s_1<x_1^*<0$.
Thus, for $x\in(-\infty,\underline{s})$,
\[
\Gamma \bar{V}_1(x)-\beta \bar{V}_1(x)+g(x)=\Phi_5(x)>\Phi_5(\underline{s}) \geqslant 0.
\]

We next check condition \eqref{eq:lowerbound-2}.
According to the expression of $\bar{V}_1$ given by \eqref{eq:valfun} and \eqref{eq:bar-V1}, the verification of condition \eqref{eq:lowerbound-2} is divided into ten cases: 1) $s_1<x_1<x_2$, 2) $S_1-Q<x_1\leqslant s_1<x_2$, 3) $S_1-Q < x_1<x_2\leqslant s_1$, 4) $\underline{s}\leqslant x_1\leqslant S_1-Q<s_1<x_2$,
5) $x_1<\underline{s}<s_1<x_2$,  6) $\underline{s}\leqslant x_1\leqslant S_1-Q<x_2\leqslant s_1$, 7) $x_1<\underline{s}< S_1-Q<x_2\leqslant s_1$, 8) $\underline{s}\leqslant x_1<x_2 \leqslant S_1-Q$, 9)  $x_1<\underline{s}\leqslant x_2\leqslant S_1-Q$, and 10)  $x_1<x_2<\underline{s}$.

First note that for $x\in[S_1-Q, \infty)$,  $\bar{V}_1(x)$ is  the same as $V_1(x)$ defined in Theorem \ref{thm2}.
In \textsf{Cases 1), 2), and 3)}, we have $x_2>x_1\geqslant S_1-Q$, thus $\bar{V}_1(x_i)=V_1(x_i)$, $i=1,2$,
and then condition \eqref{eq:lowerbound-2} for these three cases has been proven in Theorem \ref{thm2}.
We next consider \textsf{Cases 4)-10)}.

\noindent \textsf{Case 4):} $\underline{s}\leqslant x_1\leqslant S_1-Q<s_1<x_2$. In this case, we have
\begin{align}
\label{eq:barV1-barV1}
\bar{V}_1(x_2)-\bar{V}_1(x_1)
&=v_1(x_2)-\big(v_1(x_1+Q)+K_1+kQ\big)\\
&=\int_{x_1+Q}^{x_2}[v_1'(y)+k]\,\mathrm{d}y-K_1-k\cdot(x_2-x_1).\nonumber
\end{align}
With considering the discontinuity of the setup cost at $Q$, this case is further divided into two subcases.
\noindent \textsf{Case 4.1):} $x_2\leqslant x_1+Q$. In this subcase, we have $K(x_2-x_1)=K_1$ and
\[
\int_{x_1+Q}^{x_2}[v_1'(y)+k]\,\mathrm{d}y=-\int_{x_2}^{x_1+Q}[v_1'(y)+k]\,\mathrm{d}y\geqslant 0,
\]
where the inequality follows from $s_1<x_2\leqslant x_1+Q\leqslant S_1$ and $v_1'(x)\leqslant -k$ for any $x\in[s_1,S_1]$.
Then,
\[
\bar{V}_1(x_2)-\bar{V}_1(x_1)\geqslant -K_1-k\cdot(x_2-x_1)=-K(x_2-x_1)-k\cdot(x_2-x_1).
\]
\noindent \textsf{Case 4.2):}  $x_2>x_1+Q$.  Since $s_1\leqslant \underline{s}+Q\leqslant x_1+Q\leqslant S_1\leqslant \bar{S}$, we have
\[
v_1'(y)
\begin{cases}
<-k, & \text{for $x\in(\underline{s}+Q,\bar{S})$},\\
=-k, & \text{for $x=\bar{S}$},\\
>-k, & \text{for $x\in(\bar{S},\infty)$},
\end{cases}
\]
which implies that
\begin{equation}
\label{eq:int-x1+Q}
\int_{x_1+Q}^{x_2}[v_1'(y)+k]\,\mathrm{d}y\geqslant \int_{\underline{s}+Q}^{\bar{S}}[v_1'(y)+k]\,\mathrm{d}y
=v_1(\bar{S})-v_1(\underline{s}+Q)+k\cdot (\bar{S}-\underline{s}-Q)
=K_1-K_2,
\end{equation}
where the last equality follows from the definition of $\underline{s}$ in \eqref{eq:underline-s}.
Therefore, \eqref{eq:barV1-barV1}-\eqref{eq:int-x1+Q} imply that
\[
\bar{V}_1(x_2)-\bar{V}_1(x_1)\geqslant -K_2-k\cdot(x_2-x_1)=-K(x_2-x_1)-k\cdot(x_2-x_1).
\]

\noindent \textsf{Case 5):} $x_1<\underline{s}<s_1<x_2$. In this case, we have
\begin{align*}
\bar{V}_1(x_2)-\bar{V}_1(x_1)
&=v_1(x_2)-\big(v_1(\bar{S})+K_2+
k\cdot(\bar{S}-x_1)\big)\\
&=\int_{\bar{S}}^{x_2}[v_1'(y)+k]
\,\mathrm{d}y-K_2-k\cdot(x_2-x_1).
\end{align*}
The rest analysis for this case is provided by three subcases according to the order relation of $x_2$ with $\bar{S}$, $\underline{s}+Q$ and  $s_1$.

\noindent \textsf{Case 5.1):} If $x_2\geqslant \bar{S}$, we have $x_2-x_1>\bar{S}-\underline{s}\geqslant S_1-\underline{s}\geqslant Q$, and $
\int_{\bar{S}}^{x_2}[v_1'(y)+k]
\,\mathrm{d}y\geqslant 0$ by $v_1'(y)\geqslant -k$ for $y\geqslant \bar{S}$. Hence, we have
\begin{align*}
\bar{V}_1(x_2)-\bar{V}_1(x_1)\geqslant -K_2-k\cdot(x_2-x_1)=-K(x_2-x_1)-k\cdot(x_2-x_1).
\end{align*}

\noindent \textsf{Case 5.2):} If $\underline{s}+Q<x_2<\bar{S}$, we have  $x_2-x_1>\underline{s}+Q-\underline{s}=Q$, and then
\begin{align*}
\bar{V}_1(x_2)-\bar{V}_1(x_1)
&=-\int^{\bar{S}}_{x_2}[v_1'(y)+k]
\,\mathrm{d}y-K_2-k\cdot(x_2-x_1)\\
&\geqslant-K_2-k\cdot(x_2-x_1)\\
&= -K(x_2-x_1)-k\cdot(x_2-x_1),
\end{align*}
where the inequality follows from that $v_1'(x)\leqslant -k$ for $x\in (\underline{s}+Q,\bar{S})\subset(s_1, \bar{S})$.

\noindent \textsf{Case 5.3):} If $s_1<x_2\leqslant \underline{s}+Q$,
we have
\begin{align*}
\bar{V}_1(x_2)-\bar{V}_1(x_1)
&=-\int^{\bar{S}}_{x_2}[v_1'(y)+k]
\,\mathrm{d}y-K_2-k\cdot(x_2-x_1)\\
&\geqslant -\int^{\bar{S}}_{\underline{s}+Q}[v_1'(y)+k]
\,\mathrm{d}y-K_2-k\cdot(x_2-x_1)\\
&=-(K_1-K_2)-K_2-k\cdot(x_2-x_1)\\
&=-K_1-k\cdot(x_2-x_1)\\
&\geqslant -K(x_2-x_1)-k\cdot(x_2-x_1),
\end{align*}
where the first inequality follows from $v_1'(x)\leqslant -k$ for any $x\in(s_1, \underline{s}+Q]$, the second equality follows from the definition of $\underline{s}$ in \eqref{eq:underline-s}, i.e.,
$\big(v_1(\underline{s}+Q)+K_1+kQ\big)-\big(v_1(\bar{S})+K_2+k\cdot(\bar{S}-\underline{s})\big)=0$.

\noindent \textsf{Case 6):} $\underline{s}\leqslant x_1\leqslant S_1-Q<x_2\leqslant s_1$. In this case, we have $x_1+Q\in[\underline{s}+Q,S_1]\subseteq[s_1,S_1]$ and $x_2-x_1\leqslant s_1-\underline{s}\leqslant Q$. Thus,
\begin{align*}
\bar{V}_1(x_2)-\bar{V}_1(x_1)
&=\big(v_1(S_1)+K_1+k\cdot(S_1-x_2)\big)-
\big(v_1(x_1+Q)+K_1+kQ\big)\\
&=\int^{S_1}_{x_1+Q}[v_1'(y)+k]
\,\mathrm{d}y-k\cdot(x_2-x_1)\\
&\geqslant \int^{S_1}_{s_1}[v_1'(y)+k]
\,\mathrm{d}y-k\cdot(x_2-x_1)\\
&=-K_1-k\cdot(x_2-x_1)\\
&=-K(x_2-x_1)-k\cdot(x_2-x_1),
\end{align*}
where the inequality follows from $v_1'(x)\leqslant -k$ for any $x\in[s_1,S_1]$.

\noindent \textsf{Case 7):} $x_1<\underline{s}< S_1-Q<x_2\leqslant s_1$. In this case, we have
\begin{align*}
\bar{V}_1(x_2)-\bar{V}_1(x_1)
&=\big(v_1(S_1)+K_1+k\cdot(S_1-x_2)\big)-\big(v_1(\bar{S})+K_2+k\cdot(\bar{S}-x_1)\big)\\
&=-\int_{S_1}^{\bar{S}}[v_1'(y)+k]\,\mathrm{d}y
+K_1-K_2-k\cdot(x_2-x_1)\\
&\geqslant -K_1-k\cdot(x_2-x_1)\\
&\geqslant-K(x_2-x_1)-k\cdot(x_2-x_1),
\end{align*}
where the first inequality follows from $v_1'(x)\leqslant -k$ for $x\in[S_1,\bar{S}]$
and $2K_1\geqslant K_2$.

\noindent \textsf{Case 8):} $\underline{s}\leqslant x_1<x_2 \leqslant S_1-Q$. In this case, we have $s_1\leqslant x_1+Q<x_2+Q\leqslant S_1$ and $x_2-x_1\leqslant S_1-s_1\leqslant Q$. Hence,
\begin{align*}
\bar{V}_1(x_2)-\bar{V}_1(x_1)
&=v_1(x_2+Q)+K_1+kQ-(v_1(x_1+Q)+K_1+kQ)\\
&=\int_{x_1+Q}^{x_2+Q}[v_1'(y)+k]
\,\mathrm{d}y-k\cdot(x_2-x_1)\\
&\geqslant\int_{s_1}^{S_1}[v_1'(y)+k]
\,\mathrm{d}y-k\cdot(x_2-x_1)\\
&=-K_1-k\cdot(x_2-x_1)\\
&=-K(x_2-x_1)-k\cdot(x_2-x_1),
\end{align*}
where the inequality follows from $v_1'(x)\leqslant -k$ for $x\in[s_1,S_1]$.

\noindent \textsf{Case 9):} $x_1<\underline{s}\leqslant x_2\leqslant S_1-Q$. In this case, we have $s_1\leqslant x_2+Q\leqslant S_1\leqslant \bar{S}$. Hence,
\begin{align*}
\bar{V}_1(x_2)-\bar{V}_1(x_1)
&=\big(v_1(x_2+Q)+K_1+kQ\big)-\big(v_1(\bar{S})+K_2+k\cdot(\bar{S}-x_1)\big)\\
&=-\int_{x_2+Q}^{\bar{S}}[v_1'(y)+k]
\,\mathrm{d}y+K_1-K_2-k\cdot(x_2-x_1)\\
&\geqslant K_1-K_2-k\cdot(x_2-x_1)\\
&\geqslant -K_1-k\cdot(x_2-x_1)\\
&\geqslant-K(x_2-x_1)-k\cdot(x_2-x_1),
\end{align*}
where the first inequality holds since $x_2+Q\in [s_1, \bar{S}]$ and $v_1'(x)\leqslant -k$ for $x\in[s_1,\bar{S}]$, the second inequality follows from $2K_1\geqslant K_2$.

\noindent \textsf{Case 10):} $x_1<x_2<\underline{s}$. In this case, we have
\begin{align*}
\bar{V}_1(x_2)-\bar{V}_1(x_1)
&=\big(v_1(\bar{S})+K_2+k\cdot(\bar{S}-x_2)\big)-\big(v_1(\bar{S})+K_2+k\cdot(\bar{S}-x_1)\big)\\
&=-k\cdot(x_2-x_1)\\
&>-K(x_2-x_1)-k\cdot(x_2-x_1).
\end{align*}

Combining \textsf{Cases 1)--10)}, we know that $\bar{V}_1$ defined in \eqref{eq:bar-V1} satisfies condition \eqref{eq:lowerbound-2}.

Finally, we prove that conditions \eqref{eq:lowerbound-3}-\eqref{eq:lowerbound-4} hold for $\bar{V}_1$ defined in \eqref{eq:bar-V1}.
This directly follows from
\begin{align*}
 \bar{V}_1'(x)=
\begin{cases}
\frac{1}{\beta} g^{\prime}(x)+\frac{2}{\sigma^{2}} \frac{1}{\lambda_{1}+\lambda_{2}} \cdot\big[\frac{1}{\lambda_{1}} \int_{x}^{\infty} e^{-\lambda_{1} y} g^{\prime \prime}(y)\,\mathrm{d}y \cdot e^{\lambda_{1} x}\\
\quad+\frac{1}{\lambda_{2}}\left(A_1^*-g^{\prime}(0+)-\int_{0}^{x} e^{\lambda_{2} y} g^{\prime \prime}(y) \,\mathrm{d}y\right) e^{-\lambda_{2} x}\big], & \text { for } x\geqslant0,\\
\frac{1}{\beta} g^{\prime}(x)+\frac{2}{\sigma^{2}} \frac{1}{\lambda_{1}+\lambda_{2}} \cdot\big[\frac{1}{\lambda_{1}}\left(g^{\prime}(0+)-g^{\prime}(0-)
+\int_{x}^{\infty} e^{-\lambda_{1} y} g^{\prime \prime}(y) \,\mathrm{d}y\right) e^{\lambda_{1} x} \\
\quad+\frac{1}{\lambda_{2}}\big(A_1^*-g^{\prime}(0-)+\int_{x}^{0} e^{\lambda_{2} y} g^{\prime \prime}(y) \,\mathrm{d}y\big) e^{-\lambda_{2} x}\big],  &\text { for } x\in(s_1,0),\\
 -k, &\text { for } x\in(S_1-Q,s_1],\\
 v_1'(x+Q), &\text { for } x\in [\underline{s},S_1-Q],\\
 -k, &\text { for } x\in(-\infty,\underline{s}),
\end{cases}
\end{align*}
which is given by  \eqref{eq:vi-derivate} and \eqref{eq:valfun}.

Therefore, we have completed the verification of conditions \eqref{eq:lowerbound-1}-\eqref{eq:lowerbound-4}
for $\bar{V}_1$ defined in \eqref{eq:bar-V1}. \qed

\section*{Acknowledgments}
The authors thank the anonymous
associate editor, and the referees for their valuable suggestions
and corrections.

\bibliographystyle{ormsv080}
\bibliography{refs}
\end{document}